\documentclass[11pt,a4paper,reqno]{amsart}

\usepackage[utf8]{inputenc}
\usepackage{amsmath,amsfonts,amssymb,amscd,amsthm,calc,enumerate}
\usepackage[english]{babel}
\usepackage{yfonts,color,stmaryrd}
\usepackage{hyperref}

\newtheorem{theorem}{Theorem}[section]

\newtheorem{lemma}[theorem]{Lemma}
\newtheorem{corollary}[theorem]{Corollary}

\newtheorem{question}[theorem]{Question}
\newtheorem{conjecture}[theorem]{Conjecture}

\theoremstyle{definition}
\newtheorem{definition}[theorem]{Definition}

\newcommand{\la}{\langle}
\newcommand{\ra}{\rangle}
\newcommand{\tuple}[1]{\langle #1 \rangle}

\newcommand{\cmp}[1]{\overline #1}
\newcommand{\set}[1]{\{ #1 \}}
\newcommand{\bigset}[1]{\big\{ #1 \big\}}
\newcommand{\compr}[2]{\{ #1 \,\mid\, #2 \}}
\newcommand{\restr}{\mbox{\raisebox{.5mm}{$\upharpoonright$}}}

\DeclareMathOperator{\cons}{\triangleright}
\DeclareMathOperator{\snoc}{\triangleleft}
\DeclareMathOperator{\ran}{ran}
\DeclareMathOperator{\dom}{dom}
\DeclareMathOperator{\say}{say}
\DeclareMathOperator{\ask}{ask}
\DeclareMathOperator{\try}{try}

\DeclareMathOperator{\terminates}{\!\mathbin{\downarrow}}
\DeclareMathOperator{\diverges}{\!\mathbin{\uparrow}}

\renewcommand{\le}{\leqslant}
\renewcommand{\leq}{\leqslant}
\renewcommand{\geq}{\geqslant}
\renewcommand{\nsim}{\not\sim}

\newcommand{\ex}{\exists}
\newcommand{\vph}{\varphi}

\newcommand{\A}{\mathcal{A}}
\newcommand{\B}{\mathcal{B}}
\newcommand{\E}{\mathcal{E}}
\newcommand{\G}{\mathcal{G}}

\newcommand{\K}{\mathcal{K}}
\newcommand{\eff}{\mathrm{eff}}
\newcommand{\PP}{\mathcal{P}}

\begin{document}

\title{Embeddings between partial combinatory algebras}
\date{\today}

\author[A. Golov]{Anton Golov}
\address[Anton Golov]{Radboud University Nijmegen\\
Department of Mathematics\\
P.O. Box 9010, 6500 GL Nijmegen, the Netherlands}
\email{agolov@science.ru.nl}

\author[S. A. Terwijn]{Sebastiaan A. Terwijn}
\address[Sebastiaan A. Terwijn]{Radboud University Nijmegen\\
Department of Mathematics\\
P.O. Box 9010, 6500 GL Nijmegen, the Netherlands}
\email{terwijn@math.ru.nl}

\begin{abstract}
  Partial combinatory algebras are algebraic structures that serve as 
  generalized models of computation. 
  In this paper, we study embeddings of pcas. 
  In particular, we systematize the embeddings between relativizations of
  Kleene's models, of van Oosten's sequential computation model, and of
  Scott's graph model, showing that an embedding between two relativized models
  exists if and only if there exists a particular reduction between the oracles.
  We obtain a similar result for the lambda calculus,
  showing in particular that it cannot be embedded in Kleene's first model.
\end{abstract}



\maketitle

\section{Introduction}

Partial combinatory algebras were introduced by Feferman~\cite{Feferman}
to study generalized notions of computability. 
Feferman's motivation was his study of predicative theories 
(explicit mathematics), but the interest in pcas extends to well beyond that. 
For the role of pcas in constructive mathematics, the reader may consult 
the monographs by Beeson~\cite{Beeson} and 
Troelstra and van Dalen~\cite{TroelstravanDalenII}.
In particular, pcas have been used as a basis for models of 
constructive set theory, see 
McCarty~\cite{McCarty},
Rathjen~\cite{Rathjen}, and 
Frittaion and Rathjen~\cite{FrittaionRathjen}.
They also play an important role in the theory of realizability, 
as documented in van Oosten~\cite{vanOosten}.

Combinatory algebras (i.e.\ the total versions of pcas) were studied 
long before Feferman, and even predate the lambda calculus, to which 
they are closely connected. 
For the connections between these subjects we refer the reader to   Barendregt~\cite{Barendregt}. 
Even though pcas were not defined explicitly in that context, 
the notion of partiality was studied there extensively as well. 
Given the connection with lambda calculus, there has been an obvious 
interest in pcas from this direction. 
For example, the topic of completions of pcas was studied by 
Klop and Bethke and others, cf.\ \cite{Bethke} and \cite{BethkeKlopVrijer1995}.

In this paper we study embeddings between pcas. 
Such embeddings have been studied before, with various notions 
of embedding (explained below). 
Though we also consider uncountable pcas, our results concentrate on 
countable pcas, some of which are derived from the 
uncountable examples. 
Well-studied examples of pcas are Kleene's first and second models 
$\K_1$ and $\K_2$,
van Oosten's sequential computation model $\B$~\cite{vanOostenSeqComp},
and Scott's graph model $\G$~\cite{Scott},
as well as the relativized versions of these.
We relate the existence of embeddings between relativizations of these pcas with
the existence of Turing and enumeration reductions between the oracles.

A partial combinatory algebra (pca) $\A$ is a set $A$ with a partial binary operation 
$\mathbin{\cdot}$ satisfying a particular completeness condition.
Every pca has elements representing the combinators $\mathrm{s}$ and $\mathrm{k}$ 
(cf.\ Definition~\ref{def:pca} below), 
and the existence of such elements is sufficient for $\A$ to be a pca.

Given pcas $\A = (A, \mathbin{\cdot_A})$ and $\B = (B, \mathbin{\cdot_B})$, 
a natural notion of a pca morphism $\A \to \B$ is a function 
$f : A \to B$ that preserves the application when it is defined.
An \emph{embedding} $\A \hookrightarrow \B$ is such a function that is also injective.

This notion of embedding has previously been studied by Bethke~\cite{Bethke},
who showed that every pca can be embedded in a graph model,
and by Shafer and Terwijn~\cite{ShaferTerwijn},
who showed a number of embeddings between relativizations of Kleene's
first and second models.
In the context of completions, when application in $\B$ is required to be total,
embeddings in this sense have been studied by Asperti and Ciabattoni~\cite{AspertiCiabattoni}.
Also in the context of completions, 
Bethke, Klop, and de Vrijer~\cite{BethkeKlopVrijer1995}
studied a stronger notion of embedding, 
where the function $f : A \to B$ is additionally required
to map $\mathrm{s}$ and $\mathrm{k}$ to the corresponding chosen elements in $\B$.

A more general notion of a map $\A \to \B$ was given by Longley~\cite{Longley}
in the form of an applicative morphism.
An applicative morphism $\gamma$ is not required to preserve application;
instead, there must exist a term in $\B$ that simulates
the application on the image of $\gamma$.
Additionally, applicative morphisms are not required to be single-valued.
This notion of morphism is useful in the context of realizability, cf.\ 
van Oosten~\cite{vanOosten}.

The following two results relating relativizations of $\K_1$ and $\K_2$ 
were proven in Shafer and Terwijn~\cite{ShaferTerwijn}:

\begin{theorem}[{\cite[Proposition~8.2]{ShaferTerwijn}}]
  \label{thm_ST:1}
  Let $X, Y \subseteq \omega$.  There is an embedding $\K_1^X \hookrightarrow \K_1^Y$
  if and only if $X \le_T Y$, and in this case the embedding is given by a computable function.
\end{theorem}

\begin{theorem}[{\cite[Theorem~8.5]{ShaferTerwijn}}]
  \label{thm_ST:2}
  For every $X \subseteq \omega$, there is an embedding $f : \K_1^X \hookrightarrow \K_2^X$.
  This embedding is $X$-computable, in the sense that there is an $X$-computable function $g$
  that maps $n$ to a code of $f(n)$.
\end{theorem}

We note here that the proofs of these theorems rely on the fact that the
objects of these pcas can be coded by natural numbers.
We extend this approach to obtain similar results about the other pcas
introduced above.
To this end, we generalize Ershov's notion of numbering~\cite{Ershov} to allow
for partiality, and define certain effectivity conditions on such partial numberings.
We generalize these theorems in two ways.  Taking $\K_1$ as an example, we get
\begin{enumerate}
  \item Given an arbitrary pca $\A$ equipped with a partial numbering $\gamma : \omega \to \A$,
    if the application in $\A$ can be represented using $\gamma$ in an $X$-computable way,
    which moreover respects a chosen section of $\gamma$,
    then there is an embedding $\A \hookrightarrow \K_1^X$. 
  \item Given an arbitrary pca $\A$ equipped with a partial numbering $\gamma : \omega \to \A$,
    if the application in $\A$ can be represented using $\gamma$ in a $Y$-computable way
    and there is an embedding $\K_1^X \hookrightarrow \A$,
    then $X \le_T Y$.
\end{enumerate}

We focus in particular on the relativizations
of Kleene's first and second models $\K_1$ and $\K_2$,
of van Oosten's sequential computation model $\B$ (\cite{vanOostenSeqComp}),
and of Scott's graph model $\G$ (\cite{Scott}).
For any $X$, the following sequence of embeddings holds:
\[
  \K_1^X \hookrightarrow \K_2^X \hookrightarrow \B^X \hookrightarrow \G^{X \oplus \cmp X}.
\]
The first is Theorem~\ref{thm_ST:2} 
while the second is given by the inclusion.
The third is new (Theorem~\ref{thm:B_to_G}).
Furthermore, we show that these embeddings are tight;
that is, if for some $Y$, $\K_1^X \hookrightarrow \K_2^Y$ then $\K_2^X \hookrightarrow \K_2^Y$,
and similarly for the other two (Theorems~\ref{thm:K1_to_K2_optimal}, \ref{thm:K2_to_B}, and~\ref{thm:B_to_G_optimal}).

The case where $X = \emptyset$ is of particular interest.
Let $\E$ be the class of c.e.\ sets. These form a pca by Scott 
(see section~\ref{sec:G} below). 
We have (Corollary~\ref{cor:K2eff_to_G}) the sequence of embeddings
\[
  \K_1 \hookrightarrow \K_2^\eff \hookrightarrow \B^\eff \hookrightarrow \E.
\]
We consider the question whether these embeddings are reversible, and 
show that (Corollary~\ref{cor:BX_to_K2X}, Theorem~\ref{thm:K2X_to_K1Y})
\[
  \B^X \not\hookrightarrow \K_2^X \not\hookrightarrow \K_1^X
\]
and that $\E$ does not embed in $\K_2^\eff$ (Corollary~\ref{cor:EnotintoK2eff}), 
and hence also not in~$\K_1$.
We also show that the pca of $\lambda$-calculus up to $\beta$ or $\beta\eta$-equivalence
is not embeddable in $\K_1$ (Corollary~\ref{cor:lambda_to_K1}).

Aside from comparing the embeddings between relativizations of different models,
we also study the embeddings between different relativizations of the same
model, and give necessary and sufficient conditions for when
$\K_2^X \hookrightarrow \K_2^Y$ (Theorem~\ref{thm:K2_embed_turing}),
$\B^X \hookrightarrow \B^Y$ (Theorem~\ref{thm:B_embed_turing}),
and $\G^X \hookrightarrow \G^Y$ (Theorem~\ref{thm:GX_to_GY}).

Our notation is mostly standard.  
For pcas our notation follows van Oosten~\cite{vanOosten}
and for computability theory Odifreddi~\cite{Odifreddi}.
We use calligraphic letters ($\A$, $\B$, \ldots) for arbitrary pcas
and Roman letters for the fixed elements of a pca.
We use Roman letters ($f$, $g$, \ldots) for total functions
and Greek letters ($\theta$, $\rho$, $\xi$, $\psi$, \ldots) for partial functions, 
except for $\omega$ which denotes the natural numbers.
$\tuple{\cdot, \cdot}$ denotes an effective pairing scheme 
satisfying $\tuple{0, 0} = 0$.
$D_u$ denotes the finite set with canonical code~$u$.
$\vph_e$ denotes the $e$-th partial computable (p.c.) function.
We use $\Phi_e$ to denote the $e$-th Turing functional
and $\Psi_e$ to denote the $e$-th enumeration operator.
We use $K$ to denote the halting set, consisting of those
$e$ such that $\vph_e(e)\terminates$, and for $X \subseteq \omega$ we
use $X'$ to denote the halting set relativized to $X$.

\section{Preliminaries}

\begin{definition}
  A \emph{partial applicative structure} is a set $A$ with a partial
  binary operation $\cdot$.
\end{definition}

We will denote partial applicative structures by calligraphic letters
($\A$, $\B$\ldots).
Given $\A = (A, \cdot)$ we will simply write $a \in \A$ for $a \in A$ and
$ab$ for $a \cdot b$ when no confusion can occur.
The application operator $\cdot$ associates to the left; that is,
we will take $abc$ to mean $(ab)c$.
We write $t\terminates$ do denote that all applications in the expression $t$ 
are defined, and we use Kleene equality $t\simeq s$ to denote that either 
the expressions $t$ and $s$ are both defined and equal, or both undefined. 

A partial combinatory algebra is a partial applicative structure that is combinatory complete, 
i.e.\ in which every term is represented by an element of the structure. 
By a result of Feferman~\cite{Feferman}, this is equivalent to the following definition.

\begin{definition}\label{def:pca}
  A \emph{partial combinatory algebra} is a partial applicative structure $\A$ which contains distinct
  elements $\mathrm{s}$ and $\mathrm{k}$ such that for all $a, b, c \in \A$,
  \begin{enumerate}
    \item $\mathrm{k}ab\terminates = a$,
    \item $\mathrm{s}ab\terminates$, and
    \item $\mathrm{s}abc \simeq (ac)(bc)$.
  \end{enumerate}
\end{definition}

The prototypical example of a partial combinatory algebra is Kleene's first
model $\K_1 = (\omega, \cdot)$ where $n \cdot m = \vph_n(m)$.

A pca with a total application operator is called a \emph{combinatory algebra} (ca),
and we say that such a pca is \emph{total}.

Partial combinatory algebras can represent the natural numbers and all partial
computable functions on them, as well as a pairing function $\mathrm{p}$
and projections $\mathrm{p_1}$ and $\mathrm{p_2}$.
In particular, the combinators $\mathrm{s}$ and $\mathrm{k}$ are sufficient to represent the Church
numerals $\overline{n}$, which satisfy the property that $\overline{0}fa = a$
and $\overline{n+1}fa = \overline{n}f(fa)$.
For a more complete introduction to partial combinatory algebras, we refer the
reader to van Oosten~\cite{vanOosten}.

\begin{definition}
  Given pcas $\A = (A, \cdot_A)$ and $\B = (B, \cdot_B)$, an injection $f : A \to B$
  is an \emph{embedding} if for all $a, a' \in \A$, if $aa'\terminates$
  then $f(a)f(a')\terminates = f(aa')$.
  If $\A$ embeds into $\B$ in this way we write $\A\hookrightarrow \B$.
\end{definition}

This notion of embedding is the one studied by Bethke~\cite{Bethke} and by
Asperti and Ciabattoni~\cite{AspertiCiabattoni} in the context of completions.
It was also used in Shafer and Terwijn~\cite{ShaferTerwijn} in connection 
with the ordinal analysis of pcas. 
A stronger notion has been studied by 
Bethke, Klop, and de Vrijer~\cite{BethkeKlopVrijer1995},
who required that the elements $\mathrm{s}$ and $\mathrm{k}$ in $\A$ are sent to their 
counterparts in $\B$.
Conversely, a weaker notion can be obtained by merely requiring the map
$\gamma(a) = \set{f(a)}$ to be a partial applicative morphism $\A \to \B$ 
(cf.~Longley and Normann~\cite{LongleyNormann}).
In contrast to most prior literature, our focus is on embeddings in general, rather
than exclusively on completions, which are the special case when $\B$ is a
combinatory algebra.

\section{Partial numberings and deciders}

Our key insight for generalizing Theorems~\ref{thm_ST:1} and~\ref{thm_ST:2} is that these
theorems make extensive use of the fact that both $\K_1^X$ and $\K_2^X$ can be coded using
natural numbers.
In~\cite{ShaferTerwijn}, Shafer and Terwijn implicitly use the identity coding for $\K_1^X$
and use the partial map $n \mapsto \Phi_n^X$ as a coding of $\K_2^X$.
To make these constructions explicit, we will generalize Ershov's notion of a numbering from~\cite{Ershov}
to permit partiality.

\begin{definition}
  Given a set $S$, a \emph{partial numbering} of $S$ is a surjective partial function $\gamma : \omega \to S$.
\end{definition}

If $\gamma(n) = a$ we say that $n$ is a \emph{$\gamma$-code of $a$}.
For $n, m \in \dom(\gamma)$ we write $n \nsim_\gamma m$ to indicate $\gamma(n) \neq \gamma(m)$.

Note that a numbering is not required to be computable in any sense.
On the contrary, we define our notions of computability with respect to a choice of numbering.
For instance, the embedding from Theorem~\ref{thm_ST:2} is said to be $X$-computable.
This can be understood as follows:

\begin{definition}
  Let $S$ and $T$ be sets equipped with partial numberings $\gamma : \omega \to S$ and $\delta : \omega \to T$.
  A (partial) function $\alpha : S \to T$ is \emph{(partial) $X\!$-computable with respect to $\gamma$ and $\delta$}
  if there exists a (partial) $X$-computable
  $\beta : \omega \to \omega$ such that for all $n \in \dom(\gamma)$, if $\beta(\gamma(n))\terminates$ then
  \[
    \beta(\gamma(n)) = \delta(\alpha(n)).
  \]

\end{definition}

We extend this notion of $X$-computability to $n$-ary functions in the straightforward way.
When the partial numberings are clear from context, we will omit them.

We also generalize the notion of $X$-c.e.\ relations to a partially numbered set.

\begin{definition}
  Let $X \subseteq \omega$ and $S$ a set with a partial numbering $\gamma : \omega \to S$.
  We say that a relation $R$ on $\dom(\gamma)^k$ is \emph{$X$-c.e.\ with respect to $\gamma$}
  if there exists an $X$-c.e.\ relation $R'$ on $\omega^k$ such that $R' \restr \dom(\gamma)^k = R$.
\end{definition}

In particular, we say that $\gamma : \omega \to S$ \emph{has $X$-c.e.\ inequivalence} 
if the relation $a \neq b$ on $S^2$ is $X$-c.e.\ with respect to $\gamma$.

Given a pca $\A$ and a partial numbering $\gamma : \omega \to \A$, it is natural to ask
what the complexity of the application function is with respect to $\gamma$.

\begin{definition}
  Let $\A$ be a pca with a partial numbering $\gamma : \omega \to \A$ and let $X \subseteq \omega$.
  We say that $\gamma$ is \emph{$X\!$-effectively pca-valued}
  if the application operator is partial $X$-computable; that is,
  if there exists a partial $X$-computable function $\psi$ such that for all $n, m \in \dom(\gamma)$,
  if $\gamma(n)\gamma(m)\terminates$ then
  \[
    \gamma(n)\gamma(m) = \gamma(\psi(n, m)).
  \]

  We say that $\psi$ $\gamma$-represents the application in $\A$.
\end{definition}

When $X$ is the empty set, we say that $\gamma$ is \emph{effectively pca-valued}.

Note that it can happen that for $n, m, k, l \in \dom(\gamma)$, $\gamma(\psi(n, m)) = \gamma(\psi(k, l))$
but $\psi(n, m) \neq \psi(k, l)$.

\begin{definition}
  Let $X \subseteq \omega$ and let $\A$ be a pca with an $X$-effectively pca-valued partial numbering $\gamma$.
  Let $\psi$ $\gamma$-represent the application in $\A$.
  We say $\gamma$ is a \emph{strongly $X\!$-effectively pca-valued} partial numbering if for all $n, m, k, l \in \dom(\gamma)$,
  if $\gamma(n)\gamma(m)$ and $\gamma(k)\gamma(l)$ are both defined and equal then $\psi(n, m) = \psi(k, l)$.
\end{definition}

We say that $\psi$ strongly $\gamma$-represents the application in $\A$.
Given such a $\psi$, we define a section $\eta : \A \to \omega$ of $\gamma$ as follows.
Since $\A$ is a pca, there is an element $\mathrm{i}$ such that for all $a \in \A$, $\mathrm{i}a = a$.
Fixing a $\gamma$-code $c$ of $\mathrm{i}$ and taking $a \in \A$, we let $\eta(a) = \psi(c, n)$ for
some $\gamma$-code $n$ of $a$.
Note that since $\psi$ strongly $\gamma$-represents the application in $\A$, the choice
of the code $n$ does not matter.

Our goal in representing the application is primarily to construct Turing reductions.
To this end, we generalize the notion of a characteristic function as follows.

\begin{definition}
  Let $X \subseteq \omega$, $\A$ a pca, and $a_\bot, a_\top \in \A$.
  Given a partial numbering $\gamma : \omega \to \A$, we say that
  a function $d : \omega \to \omega$ is an \emph{$(a_\bot, a_\top)$-decider for $X$}
  if for all $n \in \omega$,
  \begin{align*}
    n \not\in X &\Rightarrow \gamma(d(n)) = a_\bot\\
    n \in X &\Rightarrow \gamma(d(n)) = a_\top.
  \end{align*}
\end{definition}

The definition is motivated by the following theorem.

\begin{theorem}
  \label{thm:decider_with_ce}
  Let $X, Y \subseteq \omega$, $\A$ a pca, $a_\bot \neq a_\top \in \A$,
  $\gamma : \omega \to \A$ a partial numbering, and
  let $d$ be a $Y\!$-computable $(a_\bot, a_\top)$-decider for $X$.
  If $\gamma$ has $Y\!$-c.e.\ inequivalence, then $X \le_T Y$.
\end{theorem}

\begin{proof}
  Choose $c_\bot, c_\top \in \omega$ to be $\gamma$-codes of $a_\bot$ and $a_\top$ respectively.
  Let $R$ be the $Y\!$-c.e.\ relation such that $R \restr \dom(\gamma)^2 = \mathbin{\nsim_\gamma}$.
  To decide whether $n \in X$, compute $d(n)$ and enumerate $R$ until
  either we discover $c_\bot R d(n)$ or $c_\top R d(n)$.
  One must be the case, since $c_\bot \nsim_\gamma c_\top$, and since $c_\bot, c_\top, d(n) \in \dom(\gamma)$.
\end{proof}

An inspection of the above proof shows that the condition that $\gamma$ has $Y\!$-c.e.\ inequivalence
is stronger than required: it suffices for the relations $x \neq a_\bot$ and $x \neq a_\top$ to
be $Y\!$-c.e.\ with respect to $\gamma$.
This strengthening is not required for our results, but we will look at another generalization
in sections~\ref{sec:K2} and~\ref{sec:G}.

To conclude the preliminaries, let us consider a more general notion of deciders.

\begin{definition}
  Let $X \subseteq \omega$, $\A$ a pca, and let $A_\bot$ and $A_\top$ be subsets of $\A$.
  Let $\gamma : \omega \to \A$ be a partial numbering.
  We say that $d : \omega \to \omega$ is an $(A_\bot, A_\top)$-decider for $X$
  if for all $n \in \omega$,
  \begin{align*}
    n \not\in X &\Rightarrow \gamma(d(n)) \in A_\bot\\
    n \in X &\Rightarrow \gamma(d(n)) \in A_\top.
  \end{align*}
\end{definition}

\begin{theorem}
  \label{thm:set_decider_with_ce}
  Let $X, Y \subseteq \omega$, $\A$ a pca, and let $A_\bot$ and $A_\top$ be disjoint subsets of $\A$.
  Let $\gamma : \omega \to \A$ be a numbering and let $d$ be a $Y\!$-computable $(a_\bot, a_\top)$-decider for $X$.
  If $\gamma$ has $Y\!$-c.e.\ inequivalence and there exist $Y\!$-c.e.\ sets $C_\bot$ and $C_\top$ such that
  $\gamma(C_\bot) = A_\bot$ and $\gamma(C_\top) = A_\top$, then $X \le_T Y$.
\end{theorem}

\begin{proof}
  The proof is analogous to Theorem~\ref{thm:decider_with_ce}; to decide $n \in X$,
  enumerate $c_\bot \in C_\bot$ and $c_\top \in C_\top$ 
  and wait until we see $d(n) \nsim_\gamma c_\bot$ or $d(n) \nsim_\gamma c_\top$.
\end{proof}

\section{Kleene's first model}

Let us now consider the example of $\K_1^X$, the pca structure on $\omega$ where
$n \cdot m = \Phi^X_n(m)$.
For every $X \subseteq \omega$, the identity is a strongly $X$-effectively pca-valued
partial numbering $\omega \to \K_1^X$, since the application is itself partial $X$-computable.

In~\cite{ShaferTerwijn}, Terwijn and Shafer have shown that an embedding $\K_1^X \to \K_1^Y$
exists if and only if $X \le_T Y$.
Our new terminology gives rise to another, more general, proof of each direction.
Let us begin with the `only if' direction:

\begin{theorem}
  \label{thm:K1_decider}
  Let $X, Y \subseteq \omega$, $\A$ a pca, and $\gamma : \omega \to \A$ a
  $Y\!$-effectively pca-valued partial numbering.
  Given an embedding $f : \K_1^X \hookrightarrow \A$, there exists a
  $Y\!$-computable $(a_\bot, a_\top)$-decider for $X$ for any choice of $a_\bot,
  a_\top \in \ran(f)$.
\end{theorem}

First, let us see that this does indeed imply the aforementioned `only if' direction.

\begin{corollary}
  \label{cor:K1_to_A}
  Let $X, Y \subseteq \omega$, $\A$ a pca, and $\gamma : \omega \to \A$ a
  $Y\!$-effectively pca-valued partial numbering.
  If $\gamma$ has $Y\!$-c.e.\ inequivalence and there exists an embedding $f : \K_1^X \hookrightarrow \A$
  then $X \le_T Y$.
\end{corollary}

\begin{proof}
  This is an immediate consequence of Theorems~\ref{thm:K1_decider} and~\ref{thm:decider_with_ce}.
\end{proof}

\begin{corollary}
  If there exists an embedding $f : \K_1^X \hookrightarrow \K_1^Y$, then $X \le_T Y$.
\end{corollary}

\begin{proof}
  Take $\A = \K_1^Y$ and let $\gamma$ be the identity seen as a $Y\!$-effectively pca-valued partial numbering.
  Since inequivalence is simply inequality, it is decidable and hence also $Y\!$-c.e.
  The conclusion follows by Corollary~\ref{cor:K1_to_A}.
\end{proof}

\begin{proof}[Proof of Theorem~\ref{thm:K1_decider}]
  Let $e$ be a code of the characteristic function of $X$.
  Let $x_\bot, x_\top \in \omega$ be such that $f(x_\bot) = a_\bot$ and $f(x_\top) = a_\top$.
  By the S-$m$-$n$ theorem there exist partial $X$-computable functions $t$ and $h$ such
  that for all $i, j \in \omega$,
  \begin{align*}
    \Phi^X_{t(i)}(j) &= \Phi^X_i(j+1)\\
    h(i) &= \begin{cases}
      x_\bot & \text{if $\Phi^X_i(0) = 0$}\\
      x_\top & \text{if $\Phi^X_i(0) = 1$}\\
      \,\,\diverges & \text{otherwise}
    \end{cases}.
  \end{align*}

  Since these functions are $X$-computable, they have codes $c_t, c_h \in \omega$.

  For every $n \in \omega$, the term $c_h \cdot (\overline{n} \cdot c_t \cdot e)$ is equal to $x_\bot$ if $n \not\in X$
  and to $x_\top$ otherwise.
  This can be seen as follows: by induction on $n$ we have $\Phi^X_{\overline{n} \cdot c_t \cdot e}(x) = \Phi^X_e(x+n)$.
  It follows that by definition of $h$,
  \[
    c_h \cdot (\overline{n} \cdot c_t \cdot e) = h(\overline{n} \cdot c_t \cdot e) = \begin{cases}
      x_\bot & \text{if $\Phi^X_e(n) = 0$}\\
      x_\top & \text{if $\Phi^X_e(n) = 1$}\\
      \,\,\diverges & \text{otherwise}
    \end{cases}.
  \]

  Since $\Phi^X_e$ is a characteristic function it is total and $(0, 1)$-valued, hence
  the third case never occurs.

  It remains to construct $d$.
  Note that $c_h \cdot (\overline{n} \cdot c_t \cdot e)$ is a term in $\K_1^X$, but $d$ must give us $\gamma$-codes of elements in $\A$.
  We now use the embedding $f : \K_1^X \hookrightarrow \A$.
  Since $f$ respects application, we can write $f(c_h \cdot (\overline{n} \cdot c_t \cdot e))$ as $f(c_h) \cdot (f(\overline{n}) \cdot f(c_t) \cdot f(e))$.
  Fixing $\gamma$-codes of $f(c_h)$, $f(c_t)$, $f(e)$, $f(\mathrm{s})$, and $f(\mathrm{k})$ we can thus use $\psi$,
  the $\gamma$-representation of application in $\A$, to construct
  a $\gamma$-code $d(n)$ for $f(c_h \cdot (\overline{n} \cdot c_t \cdot e))$.\footnote{Note that $f(\overline{n})$ can be constructed this way because
  $\overline{n}$ can be expressed in terms of $\mathrm{s}$ and~$\mathrm{k}$.}
  This is indeed $Y\!$-computable, since $\psi$ is a partial $Y\!$-computable function, and $d$ is total
  since $c_h \cdot (\overline{n} \cdot c_t \cdot e)\terminates$.

  We thus have
  \[
    \gamma(d(n)) = f(c_h \cdot (\overline{n} \cdot c_t \cdot e)) = \begin{cases}
      f(x_\bot) = a_\bot & \text{if $n \not\in X$}\\
      f(x_\top) = a_\top & \text{if $n \in X$}
    \end{cases},
  \]
  as desired.
\end{proof}

The above proof and corollary contain the essential themes that will come back 
repeatedly in this paper.
When our goal is to show that an embedding $f : \A' \to \A$ gives rise to a reduction from $X$ to $Y$,
we will find a representation $e$ of $X$ in $\A'$ and then use two carefully chosen elements
$t$ (for \emph{tail}) and $h$ (for \emph{head}), to construct the term $h(\overline{n}te)$ that
depends on whether $n \in X$.
We will then use a $Y\!$-effectively pca-valued numbering $\gamma : \omega \to \A$
to construct a $\gamma$-code for the term $f(h(\overline{n}te))$, which can be done $Y\!$-computably
because it requires us to know the code of only the finitely many $\gamma$-codes of
$f(h)$, $f(t)$, etc.
The resulting function $d : \omega \to \omega$ will then suffice to construct the desired reduction.

Let us now look at the other generalization of Theorem~\ref{thm_ST:1}, where we
show that for an arbitrary pca $\A$ equipped with a suitable partial numbering $\gamma$
we can construct an embedding $\A \hookrightarrow \K_1^X$.

\begin{theorem}
  \label{thm:count_to_K1}
  Let $X \subseteq \omega$ and let $\A$ be a pca with a strongly $X$-effectively
  pca-valued partial numbering $\gamma$.
  There exists a computable embedding $f : \A \hookrightarrow \K_1^X$.
\end{theorem}

\begin{proof}
  This proof is analogous to the proof of Theorem~\ref{thm_ST:1} in~\cite{ShaferTerwijn}.

  Making use of the $S$-$m$-$n$ and recursion theorems, define $g$ to be an
  injective function satisfying
  \[
    \Phi_{g(n)}^X(m) =
    \begin{cases}
      g(\psi(n, k)) & \text{if $\exists k.\, g(k) = m$}\\
      \uparrow & \text{otherwise}
    \end{cases}.
  \]

  Recall that since $\gamma$ is strongly $X$-effectively pca-valued, it has a
  section $\eta$ that sends every $a \in \A$ to its chosen $\gamma$-code. 
  The embedding $f$ is now given by the composition $g \circ \eta$,
  which is injective since both $g$ and $\eta$ are injective.

  Let $a, b \in \A$ and suppose $ab\terminates$.
  Since $\gamma(\psi(\eta(a), \eta(b))) = \gamma(\eta(ab))$,
  we have $\psi(\eta(a), \eta(b)) = \eta(ab)$ by the choice of $\eta$.
  It follows that
  \[
    g(\eta(a)) \cdot g(\eta(b)) = \Phi_{g(\eta(a))}^Y(g(\eta(b))) = g(\psi(\eta(a), \eta(b))) = g(\eta(ab))
  \]
  as desired.
\end{proof}

\begin{corollary}
  If $X \le_T Y$ then there exists a computable embedding $\K_1^X \hookrightarrow \K_1^Y$.
\end{corollary}

\begin{proof}
  If $X \le_T Y$, every strongly $X$-effectively pca-valued partial numbering is also
  strongly $Y\!$-effectively pca-valued.
  Hence the identity gives us the desired numbering and the existence of the embedding
  follows from Theorem~\ref{thm:count_to_K1}.
\end{proof}

\section{The $\lambda$-calculus}

The $\lambda$-calculus naturally gives rise to a ca $\lambda$ where the
elements are $\lambda$-terms up to $\alpha\beta$-equivalence, and the
application operation is given by application of $\lambda$-terms.
We can quotient this ca by an equivalence relation to obtain a new applicative
structure, which will again be a ca as long as it is not trivial.

Let $\lambda\eta$ be the $\lambda$-terms up to $\alpha\beta\eta$-equivalence.
This is an extensional ca.

\begin{theorem}
  Let $\A$ be a pca with an effectively pca-valued partial numbering $\gamma : \omega \to \A$.
  Given an embedding $f : \lambda\eta \hookrightarrow \A$ or $f : \lambda \hookrightarrow \A$,
  there exists for some $a \in \A$ a computable $(\set{a}, \A - \set{a})$-decider for
  the halting set $K$.
\end{theorem}

\begin{proof}
  Given a natural number $n$, let $\overline{n}$ be the Church numeral $n$ seen as a $\lambda$-term.
  Let $R$ be a $\lambda$-term such that for all $n, m, s \in \omega$ and all $\lambda$-terms $x, y$,
  \[ R\,\overline{n}\,\overline{m}\,\overline{s}\, x\, y =
  \begin{cases}
    x & \text{if $\vph_{n,s}(m)\terminates$}\\
    y & \text{otherwise}
  \end{cases}.
  \]

  Using the fixed point combinator $Y$, define 
  \[
    H = \lambda x.\, Y(\lambda f\, s.\, R\, x\, x\, s\, \overline{0}\, (f \, (s + \overline{1})))(\overline{0}).
  \]

  Recall that $K$ is the halting set, the set of $n \in \omega$ such that $\vph_n(n)\terminates$.
  Note that the term $H \overline{n}$ has a normal form if and only if $n \in K$, and this holds both with respect to $\lambda$
  and $\lambda\eta$.
  By combinatory completeness, $H$ can be defined using only the $\mathrm{s}$ and $\mathrm{k}$
  combinators.
  We can thus effectively construct a $\gamma$-code $d(n)$ for $f(H\overline{n})$.
  Since $\gamma(d(n)) = f(\overline{0})$
  if and only if $n \in K$, it follows that $d$ is an $\left(\set{f(\overline{0})}, \A - \set{f(\overline{0})}\right)$-decider for $K$.
\end{proof}

\begin{corollary}
  \label{cor:lambda_to_K1}
  For any $X \subseteq \omega$, if $\lambda\eta$ or $\lambda$ embeds in $\K_1^X$ then $K \le_T X$.
\end{corollary}

\begin{proof}
  We use the identity numbering for $\K_1^X$, which is $X$-effectively pca-valued.
  It is easy to verify that the conditions for Theorem~\ref{thm:set_decider_with_ce} are satisfied.
\end{proof}

This result provides the following corollary: for a pca $\A$, there cannot exist
both an embedding $\lambda \hookrightarrow \A$ (respectively $\lambda\eta
\hookrightarrow \A$) and an embedding $\A \hookrightarrow \K_1$.
It follows that to demonstrate that $\A$ does not embed in $\K_1$, it suffices
to show that $\A$ is a faithful model of the $\lambda$-calculus (respectively
$\lambda\eta$-calculus); that is, a model which does not identify any
$\alpha\beta$-inequivalent (respectively $\alpha\beta\eta$-inequivalent) terms.

\section{Kleene's second model}
\label{sec:K2}

We will use the coding of $\K_2$ from Shafer and Terwijn~\cite{ShaferTerwijn}, 
in which application in $\K_2$ is defined for all $g,h\in\omega^\omega$ by 
\begin{equation} \label{alt} 
g\cdot h = \Phi^{g\oplus h}_{g(0)},
\end{equation}
where it is understood that $g\cdot h$ is defined 
if and only if the function computed on the right is total, 
and $\Phi_e$ denotes the $e$-th Turing functional, 
viewed as a function of the oracle.  
This coding is much easier to work with than Kleene's original definition.
In Appendix~\ref{sec:coding} we will show that the two codings are equivalent, 
in the sense that they embed into each other.

Note that if we allow the elements to be partial functions, this coding gives
van Oosten's model $\B$ for sequential computations from~\cite{vanOostenSeqComp},
up to the same differences in coding as above.
Since the application is (up to totality) defined the same in both cases,
the inclusion is an embedding.

For a set of natural numbers $X$, let $\K^X_2$ denote the restriction of $\K_2$ to
those $g \in \K_2$ that are $X$-computable.
Since $\mathrm{s}$ and $\mathrm{k}$ in $\K_2$ can be chosen to be computable,
$\K^X_2$ is a pca for every $X$.
Similarly, we denote the restriction of $\B$ to the partial $X$-computable
functions by $\B^X$.
This too is a pca for every $X$.

\begin{theorem} \label{thm:count_to_K2}
  Let $X \subseteq \omega$ and let $\A$ be a pca with a strongly $X$-effectively
  pca-valued partial numbering $\gamma$.
  Then $\A$ is embeddable in $\K_2^X$ and hence in $\K_2$.
\end{theorem}

Note that by Theorem~\ref{thm:count_to_K1}, it would suffice to show that for
all $X \subseteq \omega$ there exists an embedding $\K_1^X \hookrightarrow \K_2^X$.
However, since the proof is essentially the same as the general case, we
choose to present this result directly.

\begin{proof}
    We define an injective function $f:\A \rightarrow \K^X_2$ such that 
    \begin{equation} \label{emb}
      a \cdot b\terminates = c 
      \Longrightarrow f(a) \cdot f(b)\terminates = f(c). 
    \end{equation}
    What makes \eqref{emb} hard to realize is that $f(a) \cdot f(b)$ has 
    to produce the {\em same\/} $f(c)\in\omega^\omega$, for all the 
    different $a$ and $b$ such that $a\cdot b = c$ in $\A$.

    Fortunately, the embedding $f$ does not need to be effective, and since 
    $f(a)\in\omega^\omega$ is an infinite object, we can code all the 
    necessary information into $f(a)$ as follows. 
    Recall that $\eta$ is the section of $\gamma$ that exists
    since $\gamma$ is strongly $X$-effectively pca-valued.
    Let $\psi$ be the $\gamma$-representation of the application in $\A$
    and let $\psi_s$ be the $s$-step approximation of $\psi$.

    For every $a\in\A$ we define:
    \begin{align*}
    f(a)(0) &= e \text{ (a code defined below)} \\
    f(a)(1) &= \eta(a) \\
    f(a)(2+\tuple{n, m, s}) &= 
    \begin{cases}
    1+\psi_s(n, m) &\text{if $\psi_s(n, m)\terminates$,} \\
    0 &\text{otherwise.}
    \end{cases}
    \end{align*}
    Note that for every $i\neq 1$, the information $f(a)(i)$ is the same for all~$a$.  This
    information is computable from $X$, hence $f(a) \in \K^X_2$.

    Next we use the recursion theorem to define the code $e$ such that 
    \begin{align*}
    \Phi^{g\oplus h}_e(0) &= e \\
    \Phi^{g\oplus h}_e(1) &= k \text{ (where $k$ is as explained below)}\\
    \Phi^{g\oplus h}_e(i) &= g(i) \text{ ($i>1$)}
    \end{align*}
    The code $k$ in the second line is computed as follows. First extract 
    the codes $n = g(1)$ and $m=h(1)$. Then search for the least $s$ such that
    $g(2+\la n,m,s\ra) = 1+k$, and output $k$. 
    If $k$ is not found the value is undefined. 

    Now we observe that if $a\cdot b\terminates =c$ in $\A$, 
    and $n = \eta(a), m = \eta(b), k = \eta(c)$, then  
    $$
    f(c)(1) = (f(a)\cdot f(b))(1) = \Phi^{f(a)\oplus f(b)}_e(1) = k, 
    $$
    because $\Phi_e$ extracts $n$ and $m$ from $f(a)$ and $f(b)$, and 
    then searches for the minimal $s$ such that 
    $f(a)(2+\la n,m,s\ra) = 1+k$. 
    This shows that \eqref{emb} is satisfied. 
\end{proof}

\begin{corollary}
  Every countable pca $\A$ embeds in $\K_2$.
\end{corollary}

\begin{proof}
  Recall that by definition, every pca is non-empty.
  It is easy to show that every pca is infinite (cf.~van Oosten~\cite{vanOosten}).
  Take a bijective function $\gamma : \omega \to \A$ and let $\psi(n, m) = k$ if $\gamma(n)\gamma(m) = \gamma(k)$.
  Let $X$ be the graph of $\psi$.
  By Theorem~\ref{thm:count_to_K2}, $\A$ embeds in $\K_2^X$ and hence in $\K_2$.
\end{proof}

In order to better understand embeddings between $\K_2^X$ and $\K_2^Y$, we look at embeddings out of $\K_2^X$ into
an arbitrary partially numbered pca.

\begin{theorem}
  \label{thm:K2_decider}
  Let $X, Y \subseteq \omega$, $\A$ a pca, and $\gamma : \omega \to \A$ a $Y\!$-effectively pca-valued partial numbering.
  Given an embedding $f : \K_2^X \hookrightarrow \A$ and any $g_\bot, g_\top
  \in \K_2^X$, there is a $Y\!$-computable $(f(g_\bot), f(g_\top))$-decider $d$ for $X$.
\end{theorem}

Note that the existence of an embedding $f' : \K_1^X \hookrightarrow \K_2^X$ and Theorem~\ref{thm:K1_decider}
are together not enough to conclude this, since we may have that $g_\bot$ or $g_\top$ falls outside $\ran(f')$.

\begin{proof}
  Define $t, h \in \K_2^X$ by
  \begin{align*}
    (tx)(n) &= x(n+1)\\
    (hx)(n) &= \begin{cases}
      g_\bot(n) & \text{if $x(0) = 0$}\\
      g_\top(n) & \text{if $x(0) = 1$}\\
      \diverges & \text{otherwise}
    \end{cases}.
  \end{align*}

  Note that $\chi_X$, the characteristic function of $X$, is itself an element of $\K_2^X$.
  We can thus construct a $Y\!$-computable $d$ such that for all $n \in \omega$,
  $\gamma(d(n)) = f(h(\overline{n}t\chi_X))$.
  We see that $h(\overline{n}t\chi_X) = g_\bot$ if $n \not\in X$ and $h(\overline{n}t\chi_X) = g_\top$ otherwise,
  as desired.
\end{proof}

As in the case of the $\lambda$-calculus, this immediately gives us a result on the embeddability
of $\K_2^X$ in $\K_1^Y$.
However, we will see in Theorem~\ref{thm:K2X_to_K1Y} that a more complicated 
construction allows for a stronger result.

\begin{corollary}
  If there exists an embedding $\K_2^X \hookrightarrow \K_1^Y$ then $X \le_T Y$.
\end{corollary}

\begin{proof}
  This follows directly from Theorems~\ref{thm:K2_decider} and~\ref{thm:decider_with_ce}.
\end{proof}

In order to also say something about embeddings $\K_2^X \hookrightarrow \K_2^Y$, we require a lemma on the
complexity of $\nsim_\gamma$.

\begin{lemma}
  \label{lem:K2_numbering}
  Let $X \subseteq \omega$.
  The numbering $\gamma(n) = \Phi^X_n$, restricted to those $n \in \omega$ such that $\gamma(n)$ is total,
  is an effectively pca-valued partial numbering, which we will denote by $\gamma_t$.
  This numbering has $X\!$-c.e.\ inequivalence.
\end{lemma}

\begin{proof}
  Let $n, m \in \omega$ such that $\gamma_t(n)\terminates, \gamma_t(m)\terminates$.
  By the $S$-$m$-$n$ theorem, there is a computable function $g$ such that 
  \[
    \Phi^X_{g(n, m)}(k) = (\Phi^X_n \cdot \Phi^X_m)(k),
  \]
  since the right-hand side is $X$-computable.
  This $g$ $\gamma$-represents application in $\K_2^X$,
  showing that $\gamma_t$ is indeed an effectively pca-valued partial numbering.
  
  Let $R$ be the binary relation on $\omega$ defined by
  \[
    nRm \Leftrightarrow \exists k.\, \exists s.\, \Phi^X_{n,s}(k) \terminates \wedge \Phi^X_{m,s}(k)\terminates \wedge \Phi^X_{n,s}(k) \neq \Phi^X_{m,s}(k).
  \]

  This is $X$-c.e., and $R \restr \dom(\gamma)^2 = \mathbin{\nsim_{\gamma_t}}$,
  hence $\gamma_t$ has $X$-c.e.\ inequivalence.
\end{proof}

\begin{theorem}
  \label{thm:K2_embed_turing}
  Let $X, Y \subseteq \omega$.  There exists an embedding $\K^X_2 \to \K^Y_2$
  if and only if $X \le_T Y$.
\end{theorem}

\begin{proof}
  If $X \le_T Y$ then every $X$-computable function is in fact $Y\!$-computable and
  hence $\K^X_2$ is a sub-pca of $\K^Y_2$ and the inclusion is an embedding.

  Conversely, let $d : \omega \to \omega$ be the decider constructed in Theorem~\ref{thm:K2_decider}.
  By Lemma~\ref{lem:K2_numbering}, $\K^Y_2$ has a $Y\!$-effectively pca-valued partial numbering $\gamma_t$,
  and hence by Theorem~\ref{thm:decider_with_ce}, since inequivalence in $\K_2^Y$ is $Y\!$-c.e.,
  we have $X \le_T Y$.
\end{proof}

\begin{theorem}
  \label{thm:K1_to_K2_optimal}
  Let $X, Y \subseteq \omega$.  There exists an embedding $\K_1^X \to \K_2^Y$ if and only if
  $X \le_T Y$.
\end{theorem}

\begin{proof}
  The ``if'' direction is simply Theorem~\ref{thm_ST:2} from Shafer and Terwijn~\cite{ShaferTerwijn},
  together with the fact that if $X \le_T Y$ then $\K_2^X \hookrightarrow \K_2^Y$ via the inclusion.

  By Lemma~\ref{lem:K2_numbering}, $\K^Y_2$ has a $Y\!$-effectively pca-valued partial numbering $\gamma$
  which has $Y\!$-c.e.\ inequivalence.
  The result follows by Corollary~\ref{cor:K1_to_A}.
\end{proof}

Let us return to the question of when there is an embedding $\K_2^X \hookrightarrow \K_1^Y$.

\begin{lemma}
  \label{lem:K2_ce_decider}
  Let $X, Y \subseteq \omega$ and let $Z$ be an $X\!$-c.e.\ set.
  Let $\A$ be a pca and $\gamma : \omega \to \A$ be a $Y\!$-efficiently pca-valued partial numbering.
  If there exists an embedding $f : \K_2^X \hookrightarrow \A$, then there exists an $a \in \A$
  and a $Y\!$-computable $(\set{a}, \A - \set{a})$-decider $d$ for $Z$.
\end{lemma}

\begin{proof}
  To begin, we encode the set $X$ into an element $g$ of $\K_2^X$.
  We use $g(0)$ to represent a code $e$, which we will define shortly.
  Define
  \begin{align*}
    g(0) &= e\\
    g(n+1) &= \begin{cases}
      0 & \text{if $n \not\in X$}\\
      1 & \text{if $n \in X$}
    \end{cases}.
  \end{align*}

  Since $Z$ is $X$-c.e., it has an $s$-step $X$-computable approximation $Z_s$.
  We define $e$ to be a code such that for all $h \in \K_2^X$,
  \[
    \Phi_e^{g \oplus h}(s) = \begin{cases}
      0 & \text{if $h(0) \not\in Z_s$}\\
      1 & \text{if $h(0) \in Z_s$}\\
    \end{cases}.
  \] 

  The family of functions $h_n$, given by $h_n(0) = n$ and $h_n(k+1) = 0$,
  can be uniformly constructed out of finitely many elements of $\K_2^X$;
  hence we can $Y\!$-computably find a $\gamma$-code for every $f(h_n)$.

  Let $d(n)$ be a $\gamma$-code of $f(g \cdot h_n)$.
  Note that $g \cdot h_n$ is the constant zero function exactly when $n \not \in Z$.
  Hence taking $a$ to be the image of the constant zero function under $f$, we get exactly the desired result.
\end{proof}

\begin{theorem}
  \label{thm:K2X_to_K1Y}
  If there is an embedding $\K_2^X \hookrightarrow \K_1^Y$ then $X' \le_T Y$.
\end{theorem}

\begin{proof}
  By Lemma~\ref{lem:K2_ce_decider} there is a $Y\!$-computable $d$ and a $c \in \omega$ such that $d(n) = c$
  if and only if $n \not \in X'$, which suffices.
\end{proof}

Let us now consider the case of $\B$.
Note that there is no need to generalize Theorem~\ref{thm:count_to_K2} since there is an
embedding $\K_2^X \hookrightarrow \B^X$ given by the inclusion, and so the corresponding
result is given by composition.
We therefore focus on embeddings out of $\B^X$.

\begin{theorem}
  \label{thm:B_decider}
  Let $X, Y \subseteq \omega$, $\A$ a pca, and $\gamma : \omega \to \A$ a $Y\!$-effectively pca-valued partial numbering.
  Given an embedding $f : \B^X \hookrightarrow \A$ and any $\theta_\bot, \theta_\top
  \in \B^X$, there is a $Y\!$-computable $(f(\theta_\bot), f(\theta_\top))$-decider $d$ for $X$.
\end{theorem}

\begin{proof}
  The proof is identical to that of Theorem~\ref{thm:K2_decider};
  we can perform the same construction, generalizing $h$ to return codes of partial functions.
\end{proof}

The statement corresponding to Theorem~\ref{thm:K2_embed_turing} also holds for $\B$,
but the situation is more complicated.
The map $\gamma(n) = \Phi^X_n$ is an effectively pca-valued partial numbering, but
the relation $\nsim_\gamma$ for $\B^X$ is not $X$-c.e.: given codes of $\theta \subseteq \rho \in \B^X$
we cannot tell whether $\dom(\theta) = \dom(\rho)$.
In order to avoid this problem we must ensure that we find a decider whose image
consists of two incomparable elements.
Fortunately, this is always possible.

\begin{lemma}
  \label{lem:monotonic_not_chain}
  Let $\A$ and $\A'$ be pcas and let $f : \A \hookrightarrow \A'$ be an embedding.
  If $\mathbin{\le}$ is a partial order on $\A'$ under which the application is monotonic,
  then $\ran(f)$ contains incomparable elements.
\end{lemma}

\begin{proof}
  Recall that every pca has a pairing function $\mathrm{p}$ and projections $\mathrm{p_1}$ and $\mathrm{p_2}$.
  Suppose for a contradiction, without loss of generality, that $f(\mathrm{p_1}) \subset f(\mathrm{p_2})$.
  Let $a = \mathrm{p} \cdot \mathrm{p_2} \cdot \mathrm{p_1}$, then
  \[
    f(\mathrm{p_2}) = f(\mathrm{p_1}a)
      = f(\mathrm{p_1})f(a) \subseteq f(\mathrm{p_2})f(a)
      = f(\mathrm{p_2}a)
      = f(\mathrm{p_1}),
  \]
  giving a contradiction.
\end{proof}

\begin{theorem}
  \label{thm:B_embed_turing}
  Let $X, Y \subseteq \omega$.
  There exists an embedding $\B^X \to \B^Y$ if and only if $X \le_T Y$.
\end{theorem}

\begin{proof}
  One direction is trivial, as in the case of Theorem~\ref{thm:K2_embed_turing}.

  For the other direction, let $\gamma : \omega \to \B^Y$ be the mapping $\gamma(n) = \Phi^Y_n$.
  This is a $Y\!$-effectively pca-valued partial numbering, with the construction being the same
  as in Lemma~\ref{lem:K2_numbering}.

  By Theorem~\ref{thm:B_decider} there is an $(f(\theta_\bot), f(\theta_\top))$-decider $d$ for $X$,
  where $f(\theta_\bot)$ and $f(\theta_\top)$ can be taken to be incomparable.
  Pick a point $k_\bot$ such that $f(\theta_\bot)(k_\bot)\terminates$ and
  either $f(\theta_\top)(k_\bot)\diverges$ or $f(\theta_\bot)(k_\bot) \neq f(\theta_\top)(k_\bot)$,
  and a point $k_\top$ where the reverse is true.
  By computing $\Phi_{d(n)}^Y(k_\bot)$ and $\Phi_{d(n)}^Y(k_\top)$, one of which terminates,
  and comparing it to $f(\theta_\bot)(k_\bot)$ or $f(\theta_\top)(k_\top)$ respectively,
  we can decide whether $n \in X$.
\end{proof}

\begin{theorem}
  \label{thm:K2_to_B}
  Let $X, Y \subseteq \omega$.
  There exists an embedding $\K_2^X \to \B^Y$ if and only if $X \le_T Y$.
\end{theorem}

\begin{proof}
  One direction is trivial, since if $X \le_T Y$ then $\K_2^X$ is a sub-pca of $\B^Y$.

  For the other direction, let $f : \K_2^X \to \B^Y$ be an embedding and pick $g_\bot, g_\top \in \K_2^X$
  such that $f(g_\bot)$ and $f(g_\top)$ are incomparable; this is possible by Lemma~\ref{lem:monotonic_not_chain}.

  Now let $d : \omega \to \omega$ be the $Y\!$-computable $(f(g_\bot), f(g_\top))$-decider given by Theorem~\ref{thm:K2_decider}.
  Since $f(g_\bot)$ and $f(g_\top)$ are incomparable, we can $Y\!$-computably decide whether $k \in X$
  by enumerating the graphs of $f(g_\bot)$ and $f(g_\top)$ until we see these disagree, as we do in the proof of
  Theorem~\ref{thm:B_embed_turing}.
\end{proof}

Let us now show that embeddings out of $\B^X$ can allow us to construct deciders for more than just $X$.
Compare this to Theorems~\ref{thm:K2X_to_K1Y} and~\ref{thm:GX_to_K2Y}.

\begin{theorem}
  \label{thm:B_ce_decider}
  Let $X, Y \subseteq \omega$ and let $\A$ be a pca with a $Y\!$-effectively pca-valued
  partial numbering $\gamma$.
  Given an embedding $f : \B^X \hookrightarrow \A$ and any $\theta
  \in \B^X$, there is a $Y\!$-computable $(f(\emptyset), f(\theta))$-decider $d$
  for every $X\!$-c.e.\ set $Z$.
\end{theorem}

\begin{proof}
  The proof is similar to that of Theorem~\ref{thm:K2_decider}.
  Let $Z$ be an $X$-c.e.\ set and let $\xi$ be an $X$-computable partial function such that $\dom(\xi) = Z$.

  Define $t, h \in \B^X$ by
  \begin{align*}
    t(\mu)(n) &= \mu(n+1)\\
    h(\mu)(n) &= \begin{cases}
      \theta(n) & \text{if $\mu(0)\terminates$}\\
      \,\,\,\diverges & \text{otherwise}\\
    \end{cases}.
  \end{align*}

  It may be tempting to think that we can take any $\theta_\bot \subseteq
  \theta_\top$ instead of $\emptyset$ and $\theta$, but this is not the case,
  since we cannot perform the computations $\mu(0)$ and $\theta_\bot(n)$ in
  parallel.

  As in the proof of Theorem~\ref{thm:K2_decider}, the term $h(\overline{n}t\xi)$
  is $\emptyset$ if $n \not\in X$ and is $\theta$ if $n \in X$.
  Letting $d(n)$ be a code of $f(h(\overline{n}t\xi))$, which can be constructed
  $Y\!$-computably using the fact $\gamma$ is a $Y\!$-effectively pca-valued partial numbering,
  we see that $d$ is indeed an $(f(\emptyset), f(\theta))$-decider for $X'$.
\end{proof}

\begin{corollary}
  \label{cor:BX_to_K2Y}
  Let $X, Y \subseteq \omega$.
  If there exists an embedding $f : \B^X \hookrightarrow \K_2^Y$ then $X' \le_T Y$.
\end{corollary}

\begin{proof}
  This follows from Theorems~\ref{thm:B_ce_decider} and~\ref{thm:decider_with_ce} since $X'$ is $X$-c.e.
\end{proof}

\begin{corollary}\label{cor:BX_to_K2X}
  There is no embedding $\B^X \hookrightarrow \K_2^X$ for any~$X$.
\end{corollary}

Since $\K_1^X$ embeds in $\K_2^X$, Corollary~\ref{cor:BX_to_K2X} holds if we replace $\K_2^X$ with $\K_1^X$.
However, in this case we can say something more precise.

\begin{theorem}
  \label{thm:BX_to_K1Xjj}
  For every $X \subseteq \omega$, there is an embedding $\B^X \hookrightarrow \K_1^{X''}$.
\end{theorem}

\begin{proof}
  Let $X \subseteq \omega$.
  There is a total computable binary function $h$ such that in $\B^X$ 
  we have for all $n, m$,
  \[
    \Phi^X_{h(n, m)} = \Phi^X_n \cdot \Phi^X_m.
  \]
  Note that using $X''$, we can decide, given $n$ and $m$, whether $\Phi^X_n = \Phi^X_m$.
  Let $h'$ be the $X''$-computable function such that for all $n, m$, $\Phi^X_{h(n, m)} = \Phi^X_{h'(n, m)}$
  and every $h'(n, m)$ is minimal.
  Using the $S$-$m$-$n$ and recursion theorems, define $g$ to be an injective computable function such that
  \[
    \Phi^{X''}_{g(n)}(k) =
    \begin{cases}
      g(h'(n, m)) & \text{if $m$ is minimal such that $g(m) = k$}\\
      \diverges & \text{otherwise.}
    \end{cases}
  \]

  We are now ready to define our embedding $f : \B^X \hookrightarrow \K_1^{X''}$.
  Given a $\psi \in \B^X$, let $f(\psi) = g(n)$ where $n$ is minimal such that $\Phi^X_n = \psi$.

  To verify that this construction is correct, let $\Phi^X_n, \Phi^X_m$ be elements of $\B^X$
  with $n, m$ minimal.
  By construction, $h'(n, m)$ is the minimal code of $\Phi^X_n \cdot \Phi^X_m$.
  Now in $\K_1^{X''}$ we have 
  \[
    g(n)\cdot g(m) = \Phi^{X''}_{g(n)}(g(m)) = g(h'(n, m)),
  \]
  as desired.
\end{proof}

\begin{theorem}
  \label{thm:BX_to_K1Y}
  For all $X, Y \subseteq \omega$, if there exists an embedding $f : \B^X \hookrightarrow \K_1^Y$
  then $X'' \le_T Y$.
\end{theorem}

\begin{proof}
  We show that the $\Pi^X_2$-complete set \[\left\{\tuple{n, m} \mid \Phi^X_n = \Phi^X_m\right\}\] is $Y$-computable.

  There is a finite set $E \subseteq \B^X$ such that every $\Phi^X_n \in \B^X$ 
  can be written as a term consisting of elements of $E$, 
  where a description of this term is computable 
  from~$n$.\footnote{To see how to achieve this, let $h_n(x) = n$
  and choose $j \in \B^X$ such that $j \cdot h_n = \Phi^X_n$. 
  Note that we can build every $h_n$ from a finite set by
  using the Church encoding of numerals, which in turn can be constructed from only $\mathrm{s}$ and $\mathrm{k}$.}

  Suppose now that there exists an embedding $f : \B^X \hookrightarrow \K_1^Y$.
  For every element $\psi \in E$ we have a corresponding code $f(\psi) \in \K_1^Y$. 
  Since $E$ is finite, we can thus construct a $Y$-computable function $g$ 
  such that $g(n) = f(\Phi^X_n)$.
  Now, since $f$ is an embedding, for any $n, m \in \omega$, $g(n) = g(m)$ 
  if and only if $\Phi^X_n = \Phi^X_m$, giving the desired result.
\end{proof}

Together, Theorems~\ref{thm:BX_to_K1Xjj} and~\ref{thm:BX_to_K1Y} precisely characterize the
embeddings $\B^X \hookrightarrow \K_1^Y$, since by Theorem~\ref{thm_ST:1} 
$Z \le_T Y$ implies that $\K_1^Z$ embeds in $\K_1^Y$.

\begin{corollary}
  For all $X, Y \subseteq \omega$, there exists an embedding $\B^X \hookrightarrow \K_1^Y$ if and only if $X'' \le_T Y$.
\end{corollary}

Note that Theorem~\ref{thm:BX_to_K1Y} can be seen as analogous to Theorem~\ref{thm:K2X_to_K1Y}
and Corollary~\ref{cor:BX_to_K2Y}.
We expect Theorem~\ref{thm:BX_to_K1Xjj} to also have such analogous results.

\begin{conjecture}
  For every $X \subseteq \omega$, there exist embeddings $\K_2^X \hookrightarrow \K_1^{X'}$ and $\B^X \hookrightarrow \K_2^{X'}$.
\end{conjecture}

Note that Theorem~\ref{thm:BX_to_K1Xjj} follows from this conjecture, since we can compose the embeddings.
All in all, this is an interesting example of a case where working with all partial computable functions
is considerably simpler than working only with the total computable functions.

To conclude, take $X \subseteq \omega$ and let us remark on the curious relation between $\K_2^X$ and $\B^X$ that these results reveal.
We know that $\K_2^X$ embeds in $\B^X$ (Theorem~\ref{thm:K2_to_B}) but not vice-versa (Corollary~\ref{cor:BX_to_K2Y}).
This suggests that $\B^X$ is in some sense `more powerful' than $\K_2^X$.
Also note that $\K_2^X$ is never isomorphic to $\B^Y$ for any $Y$
(in whatever sense of the notion of isomorphic that preserves 
embeddability).
Namely, $\K_2^X$ and $\B^Y$ never embed into the same pcas: 
If $X\leq_T Y$ then $\K_2^X \hookrightarrow \K_2^Y$ but 
$\B^Y \not\hookrightarrow \K_2^Y$, and if 
$X\not\leq_T Y$ then $\B^Y\hookrightarrow\B^Y$ but 
$\K_2^X \not\hookrightarrow \B^Y$.

\section{Scott's graph model} \label{sec:G}

In 1969, Dana Scott introduced his graph model~\cite{Scott}, 
which gives rise to the following pca structure on $\mathcal{P}(\omega)$.   

Let $D_u$ denote the canonical finite set with code~$u$.
We follow the convention that for finite sets $X \subset \omega$,
$X$ is coded by $\sum_{k \in X} 2^k$.
In particular, for every $n \in \omega$, $D_{2^n} = \set{n}$.
Let $\tuple{\cdot , \cdot}$ denote a bijective pairing scheme such that
$\tuple{0, 0} = 0$.

We define an application operator on $\mathcal{P}(\omega)$ as follows: given
$X, Y \subseteq \omega$, let
\begin{equation}\label{def:app}
    X \cdot Y = \bigset{n \mid \ex u.\, \tuple{n, u} \in X \wedge D_u\subseteq Y}.
\end{equation}
This application on $\mathcal{P}(\omega)$ gives rise to a pca that we 
will denote by $\G$. A proof that $\G$ is indeed a pca can be found in 
Scott~\cite{Scott}.\footnote{Scott does not use the terminology of pcas, 
but all the ingredients are there. Theorem~2.3 (p173) is the combinatory 
completeness of $\G$, and by Theorem~2.6 the combinators $\mathrm{k}$ and $\mathrm{s}$ 
can be defined using c.e.\ sets.}

Following Soare~\cite{Soare}, we will use $\E$ to denote the class of 
c.e.\ sets. For any set $X \subseteq \omega$ we will use
$\G^X$ to denote the least class of sets that contains $X$ and all c.e.\ sets
and is closed under application.
Note that this is not the same as the $X$-c.e.\ sets;
for example, for any c.e.\ set $X$, $\G^X = \E$.
For every $X$, $\G^X$ with the application \eqref{def:app}
is again a pca, since the combinators $\mathrm{s}$ and $\mathrm{k}$ of $\G$ 
can be taken in~$\E$ (cf.\ Scott~\cite{Scott}.)

The application operator \eqref{def:app} and the pcas $\G^X$ are closely connected 
with the structure of enumeration degrees 
(cf.\ Odifreddi~\cite[Chapter XIV]{OdifreddiII}).
Recall that $X$ is \emph{$e$-reducible\/} to $Y$ ($X \le_e Y$) if for some
c.e.\ relation $R$,
\[
  x \in X \Leftrightarrow \exists u (R(x, u) \wedge D_u \subseteq Y).
\]
Note that this is the case exactly when for some c.e.\ set $A$, $A \cdot Y = X$.
The set $\G^Y$ thus contains exactly the sets $X$ that are $e$-reducible to $Y$.
This connection is further discussed in Scott~\cite{Scott}.

Embeddings into graph models have previously been studied by Bethke
in~\cite[Chapter 6]{Bethke}, but the proof employed a different graph pca:
given a set $A$, Bethke (cf.~\cite[Chapter 3, Definition 0.2]{Bethke})
defines a family of sets
\begin{align*}
  G_0(A) &= A\\
  G_{n+1}(A) &= G_n(A) \cup \compr{(B, b)}{B \in \mathcal{P}^{\mathrm{fin}}(G_n(A)), b \in G_n(A)}\\
  G(A) &= \bigcup_{n \in \omega} G_n(A)
\end{align*}
and defines the application operator on $\mathcal{P}(G(A))$ by
\[
  X \cdot Y = \compr{b}{\exists B \subseteq Y.\, (B, b) \in X}.
\]

This gives rise to a pca, and Bethke showed that every pca $\A = (A, \cdot)$ embeds
into $\mathcal{P}(G(A))$.
In contrast, our result shows that countable pcas can be embedded into Scott's graph model $\G$.

\begin{theorem}
    \label{thm:count_to_G}
    Let $\A$ be a pca with a partial numbering $\gamma : \omega \to \A$ and let
    \begin{equation}
        X = \bigset{ \tuple{n, m, k} \mid n, m, k \in \omega \wedge \gamma(n) \cdot \gamma(m) = \gamma(k) }.
    \end{equation}
    Then $\A$ is embeddable in $\G^X$ and hence in $\G$.
\end{theorem}

\begin{proof}
    We construct an embedding $f : \A \to \G^X$. 
    Recall that $D_{2^n} = \set{n}$ in our coding.
    For every $n \in \omega$, let $n^* = \tuple{n, 1}$.
    We will ensure that for all $a \in \A$, $0 \not \in f(a)$
    and for all $n \in \omega$, if $\gamma(n) = a$ then $n^* \in f(a)$.

    We define an $(\A, \omega)$-indexed family of sets $Y_{a, i}$:
    \begin{align*}
        Y_{a, 0} &= \set{ n^* \mid \gamma(n) = a }\\
        Y_{a, i+1} &= \bigset{ \tuple{ y, 2^{n^*} }
            \mid n \in \omega, y \in Y_{a \cdot \gamma(n), i} }
    \end{align*}

    For $a \in \A$, let $f(a) = \bigcup_{i \in \omega} Y_{a, i}$.

    Recall that every pca contains an element $\mathrm{i}$ such that $ia = a$ for all $a$.
    By fixing a code $e$ such that $\gamma(e) = \mathrm{i}$, we can use $X$ to enumerate all
    pairs $n, k$ such that $\gamma(n) = \gamma(k)$, hence $Y_{a, 0} \le_e X$.
    In particular, we can find a code for such an enumeration effectively from a code for $a$.
    By induction, we can use $X$ to effectively find a code for $Y_{a, i}$ for every $a \in \A$ and $i \in \omega$,
    hence $f(a) \le_e X$.

    We will now verify that this construction works.
    For every $a \in \A$, every element of $f(a)$ is of the form $n^* = \tuple{1, n}$ with $\gamma(n) = a$
    or $\tuple{y, 2^{n^*}}$ for some $n, y \in \omega$.
    These are all tuples unequal to $\tuple{0, 0}$, hence $0 \not \in f(a)$.
    It follows that $f(a) \cdot f(b) = f(a) \cdot \set{n^* \mid n \in \gamma^{-1}(a) }$,
    since all elements of $Y_{b, i}$ for $i > 0$ are not of the form $n^*$.

    Let $a, b, c \in \A$ such that $a \cdot b = c$.
    By construction, the following are equivalent:
    \begin{align*}
        &y \in f(a) \cdot \set{n^* \mid n \in \gamma^{-1}(b) }\\
        \exists i, n \in \omega.\, &\gamma(n) = b \wedge \tuple{y, 2^{n^*}} \in Y_{a, i+1}\\
        \exists i, n \in \omega.\, &\gamma(n) = b \wedge y \in Y_{a \cdot \gamma(n), i} = Y_{a \cdot b, i} = Y_{c, i}\\
        &y \in f(c),
    \end{align*}
    as desired.
\end{proof}

For the special case of $\E = \G^\emptyset$, we in particular get the following result.

\begin{corollary}
  \label{cor:K1_to_E}
  $\K_1$ is embeddable into $\E$. 
\end{corollary}

\begin{proof}
  Taking
  \[
    X = \set{ \tuple{n, m, k} \mid n, m, k \in \omega \wedge \vph(m) = k },
  \]
  by the theorem, there exists an embedding $\K_1 \to \G^X$.
  Since $X$ is c.e., $\G^X = \E$, as desired.
\end{proof}

Since we know that $\K_1 \hookrightarrow \K_2^\eff \hookrightarrow \B^\eff$,
Corollary~\ref{cor:K1_to_E} can also be obtained using the following,
more powerful, result.

\begin{theorem}
  \label{thm:B_to_G}
  There is an embedding $\B \hookrightarrow \G$, which restricts to an embedding
  $\B^X \hookrightarrow \G^{X \oplus \cmp{X}}$ for all $X \subseteq \omega$.
\end{theorem}

\begin{proof}
  We perform a construction similar to the one above, but use the fact that elements
  of $\B$ can be approximated by finite sequences.

  We construct an embedding $f : \B \hookrightarrow \G$.
  Given $n \in \omega$, let $n^*$ be $\tuple{n, 1}$.
  As before, we ensure that $0 \not\in f(\theta)$ for all $\theta \in \B$.
  Given $u \in \omega$, let $\psi(u)(a) = b$ if there is a unique $b$ such that
  $\tuple{a, b}^* \in D_u$ and let it be undefined otherwise.
  Denote the set of codes of finite sets of elements of the form $n^*$ by $C$, that is,
  \[ C = \compr{u \in \omega}{\forall k \in D_u.\,\exists n.\, k = n^*}.\]

  We first define $f$ in stages on the finite elements\footnote{That is: partial functions with finite domain.} of $\B$,
  and then define the embedding of $\theta \in \B$ as the union of the embeddings of its finite subfunctions.
  For $\sigma \in \B$ finite and $\theta \in \B$ arbitrary, define
  \begin{align*}
    f_0(\sigma) &= \compr{\tuple{a, b}^*}{\sigma(a) = b}\\
    f_{n+1}(\sigma) &= \compr{\tuple{x, u}}{x \in f_n(\tau) \wedge u \in C \wedge \tau \subseteq \sigma \cdot \psi(u)}\\
    f(\sigma) &= \bigcup_{n \in \omega} f_n(\sigma)\\
    f(\theta) &= \bigcup_{\substack{\sigma \subseteq \theta \\ \text{$\sigma$ finite}}} f(\sigma).
  \end{align*}

  By induction on the index, it is easy to show that if $\sigma \subseteq \tau$ then $f_n(\sigma) \subseteq f_n(\tau)$,
  hence the last line of the definition is also consistent for finite $\theta$.

  Note that $f$ is injective since for all $n, k \in \omega$, $\tuple{n, k}^* \in f(\theta)$
  iff $\theta(n) = k$ allowing us to reconstruct $\theta$ from $f(\theta)$.
  Furthermore, for all $u \in C$ and $\theta \in \B$,
  if $D_u \subset f(\theta)$ then $\psi(u) \subseteq \theta$.

  To verify this construction is correct, let $\theta, \rho \in \B$.
  Suppose $x \in f(\theta) \cdot f(\rho)$.
  Then there is some $\tuple{x, u} \in f(\theta)$ such that $D_u \subseteq f(\rho)$.
  By construction, $u \in C$ and there are $\sigma \subseteq \theta$ and $\tau$ such that $\tau \subseteq \sigma \cdot \psi(u)$ and $x \in f(\tau)$.
  Since $D_u \subset f(\rho)$ it follows that $\psi(u) \subseteq \rho$ and hence $\tau \subseteq \sigma \cdot \psi(u) \subseteq \theta \cdot \rho$.
  We thus have $x \in f(\tau) \subseteq f(\theta \cdot \rho)$, as desired.

  Conversely, suppose $x \in f(\theta \cdot \rho)$.
  In particular, there is some $\tau \subseteq \theta \cdot \rho$ and $n \in \omega$ such that $x \in f_n(\tau)$.
  There exists some $\sigma \subseteq \theta$ and a $u \in C$ such that $\tau \subseteq \sigma \cdot \psi(u)$;
  this holds since for any $k \in \omega$, $(\theta \cdot \rho)(k)$ only requires a finite portion of $\theta$ and $\rho$ to compute,
  and $\tau$ is finite.
  It follows that $\tuple{x, u} \in f(\theta)$ and hence $x \in f(\theta) \cdot f(\rho)$.

  This concludes the verification.
  It remains to show that this embedding restricts to an embedding $\B^X \to \G^{X \oplus \cmp{X}}$.
  Let $\theta \in \B^X$.  We show that $f(\theta) \le_e X \oplus \cmp{X}$.

  We show that for every $n \in \omega$, there exists a computable function $g_n$ such that for every code $e$ of a finite
  partial $X$-computable function $\sigma$,
  \[
    f_n(\sigma) = \Psi^{X \oplus \cmp{X}}_{g_n(e)}.
  \]

  This is sufficient because in order to enumerate $f(\theta)$ for an arbitrary $\theta \in \B$, it suffices to enumerate the
  codes $e$ of its finite subfunctions and take the union of $g_n(e)$ for each $n$.
  This can be done effectively, since given a code $d$ of $\theta$ we can effectively find all such $e$.  

  We proceed to prove the claim by induction on $n$.
  In the base case, we can enumerate the graph of $\sigma$ given the code using $X \oplus \cmp{X}$.
  For the inductive case of $n+1$, we can enumerate all $u \in C$, all finite partial functions $\tau$, and use $g_n$ to enumerate all $x \in f_n(\tau)$.
  By enumerating $\tau$ in a way that allows us to decide its domain, the relation $\tau \subseteq \sigma \cdot \psi(u)$ can be enumerated.
\end{proof}

By Theorem~\ref{thm:B_to_G} we have that 
$\K_2 \hookrightarrow \B \hookrightarrow \G$, since the 
first embedding is just given by the inclusion. 
At this moment we do not know the status of any of the embeddings in the 
other direction.

\begin{question}
  Does there exist an embedding $\G \hookrightarrow \B$ or $\B \hookrightarrow \K_2$?
\end{question}

Theorem~\ref{thm:B_to_G} immediately gives the following result about the c.e.\ sets $\E$.

\begin{corollary}
  \label{cor:K2eff_to_G}
  There is a sequence of embeddings 
  $\K_1 \hookrightarrow \K_2^\eff \hookrightarrow \B^\eff \hookrightarrow \E$.
\end{corollary}

\begin{proof}
  The first embedding exists by Theorem~\ref{thm_ST:2}, 
  the second is by inclusion, and the third is immediate from Theorem~\ref{thm:B_to_G} 
  since $\G^{\emptyset \oplus \cmp{\emptyset}} = \G^\emptyset = \E$.
\end{proof}

In fact, the embedding given by Theorem~\ref{thm:B_to_G} is tight in the following sense.

\begin{theorem}
  \label{thm:B_to_G_optimal}
  Let $X, Y \subseteq \omega$.
  There exists an embedding $\B^X \hookrightarrow \G^Y$ if and only if $X \oplus \cmp{X} \le_e Y$.
\end{theorem}

\begin{proof}
  One direction follows from Theorem~\ref{thm:B_to_G} and the fact that $G^{X \oplus \cmp{X}} \subseteq G^Y$ when $X \oplus \cmp{X} \le_e Y$.

  For the other direction we turn to Theorem~\ref{thm:B_decider}.
  Suppose that $f : \B^X \hookrightarrow \G^Y$ is an embedding.
  By Lemma~\ref{lem:monotonic_not_chain}, there exist $\theta_\bot$ and $\theta_\top$ such that
  $f(\theta_\bot)$ and $f(\theta_\top)$ are not comparable with respect to the
  subset order on $\G^Y$.
  Let $d$ be a $Y\!$-computable $(f(\theta_\bot), f(\theta_\top))$-decider for $X \oplus \cmp{X}$.

  Pick $i \in f(\theta_\bot) - f(\theta_\top)$ and $j \in f(\theta_\top) - f(\theta_\bot)$.
  We can enumerate $X \oplus \cmp{X}$ now as follows: for each $n$, we enumerate
  $\Psi^Y_{d(n)}$ until we see $i$ or $j$.
  Note that $\Psi^Y_{d(n)}$ is by construction one of $f(\theta_\bot)$ or $f(\theta_\top)$.
  If we see $i$, then $\Psi^Y_{d(n)} \neq f(\theta_\top)$ and thus $n \not \in X$, so we enumerate $2i + 1$ into $X \oplus \cmp{X}$.
  If we see $j$, the converse is the case and we enumerate $2i$ instead.

  We hence see $X \oplus \cmp{X} \le_e Y$, as desired.
\end{proof}

We would like to prove an analogue of Theorem~\ref{thm:K2_embed_turing} for Scott's graph model,
with Turing reducibility replaced by $e$-reducibility.
However, there is a difficulty: the operations we can perform in $\G$ are monotone,
while an embedding $f : \G^X \to \G^Y$ need not be so.

\begin{theorem}
  \label{thm:G_decider}

  Let $X, Y \subseteq \omega$, $\A$ a pca, and $\gamma : \omega \to \A$ a $Y\!$-effectively pca-valued partial numbering.
  Given an embedding $f : \G^X \hookrightarrow \A$ and any $B_\bot \subseteq B_\top
  \in \G^X$, there is a $Y\!$-computable $(f(B_\bot), f(B_\top))$-decider $d$ for any $Z \le_e X$.
\end{theorem}

\begin{proof}
  Note that $Z \in \G^X$.  
  As before, we construct elements $h$ and $t$ and use the expression $h(\overline{n}tZ)$
  to denote either $B_\bot$ or $B_\top$, depending on whether $n \in Z$.

  Define, for all $A \subseteq \omega$, $n \in \omega$
  \begin{align*}
    n \in tA &\Leftrightarrow n + 1 \in A\\
    hA &= \begin{cases}
      B_\bot & \text{if $0 \not\in A$}\\
      B_\top & \text{if $0 \in A$}
    \end{cases}.
  \end{align*}

  Now let $d(n)$ be a $\gamma$-code of $f(h(\overline{n}tZ))$,
  which can be constructed from the codes of the subexpressions using $\psi$ as usual.
\end{proof}

Note that the condition $B_\bot \subseteq B_\top$ is essential for $h$ to be well-defined,
since we cannot use the absence of $0$ in $A$ to enumerate elements into $hA$.

\begin{lemma}
  The mapping $\gamma(n) = \Psi^X_n$ is an effectively pca-valued partial numbering.
\end{lemma}

\begin{proof}
  It suffices to show that application can be simulated on codes using a computable function,
  which is the case by the $S$-$m$-$n$ theorem.
\end{proof}

The relation $\mathbin{\nsim_\gamma}$ is not generally $X$-c.e.: even for the case $X = \emptyset$, given
codes of two c.e.\ sets we cannot tell whether they code the same set.
However, fixing a set $A$ and an $n \not\in A$, we can enumerate the codes of all sets $B$ such that $n \in B$.
As in the case of the van Oosten model in Theorem~\ref{thm:B_embed_turing},
the key is to find an $(a_\bot, a_\top)$-decider where $a_\bot$ and $a_\top$
are related in a suitable manner.
For our purposes it suffices that $a_\top \not\subseteq a_\bot$, hence our interest is
in those embeddings that do not reverse the inclusion order.

\begin{definition}
  An embedding $f : \G^X \to \G^Y$ is \emph{order-reversing} if for all $A, B \in \G^X$,
  whenever $A \subseteq B$, $f(B) \subseteq f(A)$.
\end{definition}

\begin{theorem}\label{thm:GXGY}
  \label{thm:GX_to_GY}
  Let $X, Y \subseteq \omega$.  There exists a non-order-reversing embedding $f
  : \G^X \to \G^Y$ if and only if $X \le_e Y$.
\end{theorem}

\begin{proof}
  One direction is trivial, since $e$-reducibility is transitive, and hence if
  $A \le_e X$ and $X \le_e Y$ then also $A \le_e Y$.
  It follows that the inclusion is an order-preserving embedding.

  For the other direction, suppose that $f : \G^X \to \G^Y$ is a non-order-reversing embedding
  and fix sets $B_\bot \subsetneq B_\top \in \G^X$ such that $f(B_\top) \not\subseteq f(B_\bot)$.
  By Theorem~\ref{thm:G_decider} there is a computable $(f(B_\bot), f(B_\top))$-decider for $X$.
  Fix now an element $k \in f(B_\top) - f(B_\bot)$, which exists by our choice of $B_\bot$ and $B_\top$.
  In order to enumerate $X$, we enumerate $\Psi_{d(n)}^Y$ looking for $k$, and enumerate $n$ into $X$
  if we find it, which shows that $X \le_e Y$.
  Note that $d$ can be taken to be computable since the map $n \mapsto \Psi_n^Y$ is effectively pca-valued.
\end{proof}

We conjecture that Theorem~\ref{thm:GXGY} can be generalized to arbitrary 
embeddings $f : \G^X \to \G^Y$. 
For the record, we pose the following question:

\begin{question}
  Do there exist $X, Y \subseteq \omega$ such that $\G^X \hookrightarrow \G^Y$ and $X \not\le_e Y$?
\end{question}

It is tempting to try to resolve the question by showing that order-reversing 
embeddings do not exist, but this is unfortunately not the case, as we now show. 

\begin{theorem}
  There exists an order-reversing embedding $\E \to \G$.
\end{theorem}

\begin{proof}
  We define $f$ inductively as follows:
  \begin{align*}
    f_0(A) &= \compr{e^*}{A \subseteq W_e}\\
    f_{n+1}(A) &= \compr{\tuple{x, 2^{e^*}}}{ e \in \omega \wedge x \in f_n(AW_e)}\\
    f(A) &= \bigcup_{n \in \omega} f_n(A).
  \end{align*}

  By induction on $n$, we can show that every $f_n$ is order-reversing.
  The base case is apparent.
  For the inductive case, note that if $A \subseteq A'$ then $AW_e \subseteq A'W_e$ hence
  by induction, $f_n(A'W_e) \subseteq f_n(AW_e)$.

  It remains to show that $f$ is an embedding.
  Injectivity follows from the definition of $f_0(A)$: we can recover $A$ by
  \[
    A = \bigcap_{e^* \in f_0(A)} W_e
  \]
  and no $f_{n+1}(A)$ contains elements of the form $e^*$.

  To see $f(AB) = f(A)f(B)$, let $x \in f(AB)$; in particular, $x \in f_n(AB)$ for some $n \in \omega$.
  Take $e$ such that $W_e = B$, then $\tuple{x, 2^{e^*}} \in f_{n+1}(A) \subset f(A)$.
  By our choice of $e$, $e^* \in f_0(B) \subset f(B)$.
  Hence $x \in f(A)f(B)$, as desired.

  For the other direction, let $x \in f(A)f(B)$.
  Then there is some $e$ such that $\tuple{x, 2^{e^*}} \in f(A)$ and $e^* \in f(B)$.
  Thus $x \in f(AW_e)$, and since $e^* \in f(B)$, $B \subseteq W_e$.
  Since $f$ is order-reversing, $f(AW_e) \subseteq f(AB)$, as desired.
\end{proof}

We can use Theorem~\ref{thm:G_decider} to clarify the relation between 
$\G^X$ and $\K_2^Y$ as follows. 
The reader may compare this to Corollary~\ref{cor:BX_to_K2Y}.

\begin{theorem}
  \label{thm:GX_to_K2Y}
  Let $X, Y \subseteq \omega$.  If there exists an embedding $f : \G^X \hookrightarrow \K_2^Y$
  then every $Z \le_e X$ is $Y\!$-computable.
\end{theorem}

\begin{proof}
  Let $\gamma_t : \omega \to \K_2^Y$ be the effectively pca-valued partial numbering from
  Lemma~\ref{lem:K2_numbering} and let $B_\bot \subsetneq B_\top \in \G^X$.
  By Theorem~\ref{thm:G_decider} there is a computable $(f(B_\bot), f(B_\top))$-decider for $Z$.
  Since $\gamma_t$ has $Y\!$-c.e.\ inequivalence, by Theorem~\ref{thm:decider_with_ce} we have $Z \le_T Y$.
\end{proof}

Theorem~\ref{thm:GX_to_K2Y} has the following two results as corollaries, which show that
Corollaries~\ref{cor:K2eff_to_G} and~\ref{cor:K1_to_E} cannot be reversed.

\begin{corollary} \label{cor:EnotintoK2eff}
  $\E$ does not embed in $\K_2^\eff$.
\end{corollary}

\begin{proof}
  Suppose an embedding $\E \hookrightarrow \K_2^\eff$ existed. Then by 
  Theorem~\ref{thm:GX_to_K2Y} the halting set $K$ would be decidable, 
  a contradiction.
  
  Note that we can also derive the corollary from Corollaries~\ref{cor:BX_to_K2X}
  and \ref{cor:K2eff_to_G}:
  By Corollary~\ref{cor:BX_to_K2X} we have that 
  $\B^\eff \not\hookrightarrow \K_2^\eff$, so if $\E \hookrightarrow \K_2^\eff$
  then by Corollary~\ref{cor:K2eff_to_G} we obtain 
  $\B^\eff \hookrightarrow \E \hookrightarrow \K_2^\eff$, 
  contradicting Corollary~\ref{cor:BX_to_K2X}.
\end{proof}

\begin{corollary}
  $\E$ does not embed in $\K_1$.
\end{corollary}

\begin{proof}
  Suppose an embedding $\E \hookrightarrow \K_1$ existed.
  Since $\K_1$ embeds in $\K_2^\eff$ (Theorem~\ref{thm_ST:1}),
  there would be an embedding $\E \hookrightarrow \K_2^\eff$, 
  contradicting Corollary~\ref{cor:EnotintoK2eff}.
\end{proof}

\appendix

\section{Equivalence of codings}
\label{sec:coding}

In section~\ref{sec:K2}, we introduced alternative codings for
Kleene's second model $\K_2$ and for van Oosten's sequential computation model $\B$.
Let us now show that these can indeed be considered to be merely different codings of
the same models, and that our results about embeddings extend to the usual codings.
Note that since our results are often about relativizations of $\K_2$ and $\B$,
it is essential that these embeddings also restrict to embeddings between the relativized
models as well.

We will use $\K_2$ and $\B$ to refer to the usual codings, as they are presented
in~\cite{vanOosten}, and $\PP$ and $\PP_\bot$ respectively to refer to our codings.
Our encodings use the same carrier sets, the set $\omega^\omega$ of total functions
$\omega \to \omega$ and the set $\omega^\omega_\bot$ of partial functions
$\omega \rightharpoonup \omega$, respectively.
For completeness, and to introduce a number of useful shorthands, we will now restate
the definitions.

Recall that we have a bijective coding scheme $\tuple{\ldots}$ on the naturals.
We will use $\vec{\alpha}$ to refer to a finite sequence of naturals, and
$\tuple{\vec{\alpha}}$ or just $\alpha$ to refer to its code.
We use $|\alpha|$ to mean the number of elements of $\alpha$,
and $\alpha_i$ to retrieve the $i$-th element of $\alpha$ (starting from 0),
seen as the code of a sequence.
Given $n, \vec{\alpha} \in \omega$, let
$n \cons \tuple{\vec{\alpha}} = \tuple{n, \vec{\alpha}}$,
$\tuple{\vec{\alpha}} \snoc n = \tuple{\vec{\alpha}, n}$,
and let $\ask(n) = 2n$ and $\say(n) = 2n+1$.
Given $x, \xi \in \omega$ we will use
$x \cons^* \xi$ to denote $((x \cons \xi_0) \cons \ldots) \cons \xi_{|\xi|-1}$; note that
here each $\xi_i$ is itself interpreted as a sequence.\footnote{This operation may resemble
concatenation but is in fact \emph{not} that: the result has length $|\xi_{|\xi|-1}| + 1$,
and for all $x, \xi, \alpha \in \omega$, $(x \cons^* \xi) \cons \alpha = x \cons^* (\xi \snoc \alpha)$.}
Given a function $f : \omega \to \omega$ and $n \in \omega$,
let $f\upharpoonright n$ be $\tuple{f(0), \ldots, f(n-1)}$.
Given $d, e \in \omega$, define $d \ast e$ as the following oracle computation for an
arbitrary oracle $\psi$:
\begin{equation}
  \label{eqn:ast}
  \Phi^\psi_{d \ast e} = \Phi^{\Phi^\psi_e}_d.
\end{equation}
That is, to execute $d \ast e$, simulate $d$ and use $e$ to compute any oracle queries that $d$ performs,
using the oracle $\psi$ for any oracle queries of $e$.
Note that $\mathbin{\ast}$ is associative, and has an identity element $e_\ast$ satisfying $\Phi^\psi_{e_\ast} = \psi$.

For the usual coding of Kleene's second model $\K_2$, it is helpful to introduce the
following notion of a limit.
Given a sequence of naturals $a_n$ and $k \in \omega$, we say $a$ \emph{converges}
to $k$ if there exists an $n$ such that $a_n = k+1$ and for all $i < n$, $a_i = 0$.
In this case, we write $\lim_{n \to \infty} a_n = k$.
Note that $n$ here does not truly ``go to infinity'', since we only look up to the
first non-zero element; however, we trust this will not cause confusion.

We extend this notion pointwise to families of functions, where given a family $f_n : \omega \to \omega$
and a function $g : \omega \to \omega$, we say $\lim_{n \to \infty} f_n = g$ when
for every $x \in \omega$, $\lim_{n \to \infty} f_n(x) = g(x)$.

Given $f, g \in \K_2$, their composition in $\K_2$ exists and is given by
\[
  f \cdot g = \lim_{n \to \infty} \lambda x.\, f(x \cons (g \upharpoonright n))
\]
if the right hand side exists (in which case it is necessarily unique).
If the limit does not exist, then $f \cdot g$ is undefined.

The usual coding of van Oosten's sequential computations $\B$ is defined in terms of \emph{dialogues},
which we will define in terms of a limit-like construction similar to the one above.

Given partial functions $\vph, \psi : \omega \rightharpoonup \omega$ and $k \in \omega$,
we say that $\vph$ $\psi$-converges to $k$ if there exists an $\alpha$ such that
$\vph(\alpha) = \say(k)$ and for every proper initial segment $\beta$ of $\alpha$,
there exists some $l$ such that $\vph(\beta) = \ask(l)$ and $\alpha_{|\beta|} = \psi(l)$.
We write $\lim_{\alpha \to \psi} \vph(\alpha) = k$ when this is the case.

We again extend this notion pointwise to functions; given $\delta_\alpha$ a family of functions, we write
$\lim_{\alpha \to \psi} \delta_\alpha = \Delta$ for some $\Delta$ iff for every $x$,
$\lim_{\alpha \to \psi} \delta_\alpha(x) \simeq \Delta(x)$.

Given two partial functions $\vph, \psi \in \B$, their composition in $\B$ is defined by
\[
  (\vph \cdot \psi)(x) = \lim_{\alpha \to \psi} \vph(x \cons \alpha).
\]

The application $\vph \cdot \psi$ is thus always defined.

Note that if $\alpha$ is the sequence witnessing that the limit exists for $x$,
the sequence $x \cons \alpha$ is what van Oosten~\cite{vanOosten} calls
a \emph{dialogue} between $\vph$ and $\psi$.

Note that these definitions suffice to give us the relativized versions of $\K_2$ and $\B$
simply by restricting the carrier set.

We are now ready to show that our codings $\PP$ and $\PP_\bot$ are mutually embeddable with
$\K_2$ and $\B$ respectively.
Given $e \in \omega$ and $\vph$ a
partial function, define
\begin{align*}
  (S_e\vph)(0) &= e\\
  (S_e\vph)(n+1) &= \vph(n).
\end{align*}
We understand $(S_e\vph)(n+1)$ to be undefined whenever $\vph(n)$ is undefined.
The embeddings of $\K_2$ and $\B$ into our codings will be given by $S_e$ for particular
choices of $e$.
The embeddings in the other direction are rather more involved.

\begin{theorem}
  \label{thm:k2_embed_p}
  There exists an $e \in \omega$ such that $S_e$ is an embedding $\K_2 \to \PP$ and
  $S_e$ restricts to an embedding $\K_2^X \to \PP^X$ for every oracle $X$.
\end{theorem}

\begin{proof}
  We define $e$ as the least natural such that for all $f, g \in \omega^\omega$ and $x, d \in \omega$,
  \begin{align*}
    \Phi^{S_df \oplus S_dg}_e(0) &= d\\
    \Phi^{S_df \oplus S_dg}_e(x+1) &= \lim_{n \to \infty} f(x \cons (g \upharpoonright n)).
  \end{align*}

  It is then easy to verify that if $f \cdot g\terminates$ then
  \begin{align*}
    (S_ef \cdot S_eg)(0) &= e\\
    (S_ef \cdot S_eg)(x+1)
     &= \Phi^{S_ef \oplus S_eg}_e(x)\\
     &= \lim_{n \to \infty} f(x \cons (g \upharpoonright n))\\
     &= (f \cdot g)(x) = S_e(f \cdot g)(x+1),
  \end{align*}
  as desired.

  Since $S_ef$ is $f$-computable, the embedding restricts to an embedding for every relativization.
\end{proof}

\begin{theorem}
  There exists an $e \in \omega$ such that $S_e$ is an embedding $\B \to \PP_\bot$ and
  $S_e$ restricts to an embedding $\B^X \to \PP_\bot^X$ for every oracle $X$.
\end{theorem}

\begin{proof}
  The proof is analogous to the proof of Theorem~\ref{thm:k2_embed_p}, with the difference
  that $e$ simulates the application in $\PP_\bot$ instead of in $\PP$.
\end{proof}

The constructions in the other direction are considerably more involved,
as the embedding of a (partial) function $\psi$ needs to encode the behavior
of the code $\psi(0)$.
For convenience, we fix a code $r \in \omega$ such that for every oracle $\psi$,
\begin{equation}
  \label{eqn:def_r}
  \Phi^\psi_r(x) = \Phi^\psi_{\psi(0)}(x)
\end{equation}
which will simplify the following construction.
Note that for $f, g \in \PP$, we have that $f \cdot g = \Phi^{f \oplus g}_r$ iff $f \cdot g\terminates$,
and similarly for $\PP_\bot$ (where the termination condition holds trivially).

In order to simplify our arguments, let us also introduce some further terminology.
When talking about a (partial) function $\vph : \omega \rightharpoonup \omega$, we will
use \emph{places} to refer to elements of its domain.
We will say that a place $x$ is \emph{coding} for $\vph$ if it can play a role when $\vph$ is applied
to another element, and \emph{non-coding} if it cannot.
Concretely, for $f \in \K_2$ a place $x \cons \alpha$ is coding if for all proper initial segments $\beta$ of $\alpha$,
$f(x \cons \beta) = 0$, and it is non-coding otherwise.
For $\vph \in \B$, a place $x \cons \alpha$ is coding if for all proper initial segments $\beta$ of $\alpha$,
$\vph(x \cons \beta)$ is of the form $\ask(k)$ for some $k$.

\begin{theorem}
  \label{thm:p_embed_k2}
  There exists an embedding $E : \PP \to \K_2$, which restricts to an embedding $\PP^X \to \K_2^X$
  for every oracle $X$.
\end{theorem}

\begin{proof}
  The proof proceeds in two steps: we first define $E : \PP \to \K_2$, and then show that it is in
  fact an embedding; that is, that for all $f, g \in \PP$, if $f \cdot g \terminates$ then
  $Ef \cdot Eg \terminates = E(f \cdot g)$.

  For the construction of $E$, we will separate $\omega$ into two sets, $X$ and $Y$.
  For every $f \in \PP$, we will use the places in $X$ to encode $f$ in $Ef$ and ensure that
  this encoding is preserved under application.
  The places in $Y$ will be given the same values in $Ef$ regardless of the choice of $f$.

  Let $\zeta$ be the all-zero sequence of length $\tuple{} + 1$.
  For $x \in \omega$ let $\dot x = \tuple{} \cons \zeta \snoc x$.
  We will ensure that for every $f \in \PP$, $\dot x$ is non-coding for $Ef$,
  so that we can set $(Ef)(\dot x) = f(x)$ and thereby encode $f$ in $Ef$.
  Given a sequence $\alpha$, let $\hat \alpha$ be the shortest sequence such that for all $k$
  with $\dot k < |\alpha|$, $\hat \alpha_k = \alpha_{\dot k}$.

  Let
  \begin{align*}
    X &= \compr{\dot{x} \cons^* \xi}{x, \xi \in \omega}\\
    Y &= \omega - X.
  \end{align*}

  For $n, \alpha \in \omega$, let $r_{n, \alpha}$ be a code for the following computation
  with an arbitrary total oracle $f$.
  On input $x$, simulate $\Phi^{f \oplus \hat{\alpha}}_r(x)$ for $|\alpha|$
  steps, where $r$ is the code defined in (\ref{eqn:def_r}).
  When $f$ is queried and returns $k$, if $k < n$ then terminate the whole computation and return $k$.
  If $k \geq n$, then use $k - n$ as the result of the oracle query.
  If $\hat{\alpha}$ is queried return the value at that position or, if it is out of bounds,
  terminate the whole computation and return $n$.
  Similarly, if the computation runs out of steps, return $n$.
  Finally, if the computation terminates with result $k$, return $k + n + 1$.

  The definition of $r_{n, \alpha}$ can be understood as follows: $r_{n, \alpha}$ is a computation
  with a left oracle that may fail in $n$ different ways (denoted by responses $0$ through $n-1$),
  and a right oracle that is partially known ($\hat\alpha$).
  Hence, $r_{n, \alpha}$ is a computation which may fail in $n+1$ different ways.

  Now define, using $\ast$ as defined in (\ref{eqn:ast}),
  \begin{align*}
    L(\tuple{}) &= e_\ast\\
    L(\alpha \cons \xi) &= r_{|\xi|, \alpha} \ast L(\xi).
  \end{align*}

  Given some possibly undefined expression $e$, let $\try(e)$ be $n+1$ if $e$ has
  value $n$, and $0$ otherwise.

  Using this, we can now define $Ef$:
  \begin{align*}
    (Ef)(\dot x \cons^* \xi) &= \Phi^f_{L(\xi)}(x) & \text{for $x, \xi \in \omega$}\\ 
    (Ef)(\tuple{}) &= 0 &\\
    (Ef)(y \cons \alpha) &= \try(\alpha_y) & \text{for $y \cons \alpha \in Y$.}
  \end{align*}

  The last line of this definition can be read as $(Ef)(y \cons \alpha) = \alpha_y + 1$
  when $y < |\alpha|$ and $(Ef)(y \cons \alpha) = 0$ otherwise.

  We immediately see that $Ef$ is in fact $f$-computable,
  from which it follows that if $E$ is an embedding $\PP \to \K_2$,
  then it restricts to an embedding $\PP^X \to \K_2^X$ for every oracle $X$.
  Moreover, for any $x \in \omega$,
  \[ (Ef)(\dot x) = \Phi^f_{L(\tuple{})}(x) = \Phi^f_{e_\ast}(x) = f(x), \]
  meaning that $E$ is injective.

  It thus remains to show that for all $f, g \in \PP$,
  if $f \cdot g \terminates$ then $Ef \cdot Eg\terminates = E(f \cdot g)$, which we will do pointwise.
  We will first show this for $Y$, since it is simpler, and then for $X$.

  Let $f, g \in \PP$, $y \in Y$, and suppose $f \cdot g \terminates$.
  Note that for all $k \in \omega$,
  \[
    \lim_{n \to \infty} \try((Eg \upharpoonright n)_k) = (Eg)(k),
  \]
  since eventually $n > k$.

  We consider two cases for $y$.
  If $y = \tuple{}$, we have
  \[
    \lim_{n \to \infty} (Ef)(\tuple{} \cons (Eg \upharpoonright n)) = \lim_{n \to \infty} \try((Eg \upharpoonright n)_{\tuple{}}) = (Eg)(\tuple{}) = 0 = E(f \cdot g)(\tuple{}).
  \]

  If $y = y' \cons \alpha$, we have
  \begin{align*}
    \lim_{n \to \infty} (Ef)((y' \cons \alpha) \cons (Eg \upharpoonright n))
      &= \lim_{n \to \infty} \try((Eg \upharpoonright n)_{y' \cons \alpha})\\
      &= (Eg)(y' \cons \alpha) = \try(\alpha_{y'}) = E(f \cdot g)(y' \cons \alpha).
  \end{align*}

  It follows that for all $y \in Y$,
  \[
    \lim_{n \to \infty} (Ef)(y \cons (Eg \upharpoonright n)) = E(f \cdot g)(y),
  \]
  as desired.

  Now let us consider the case for $X$, which is more complicated.
  Let $f, g \in \PP$, $\dot x \cons^* \xi \in X$, and assume $f \cdot g \terminates$.

  We wish to show
  \begin{align*}
    (Ef \cdot Eg)(\dot x \cons^* \xi)
    &= \lim_{n \to \infty} (Ef)((\dot x \cons^* \xi) \cons (Eg \upharpoonright n))\\
    &= \lim_{n \to \infty} (Ef)(\dot x \cons^* (\xi \snoc (Eg \upharpoonright n)))\\
    &= \lim_{n \to \infty} \Phi^f_{L(\xi \snoc (Eg \upharpoonright n))}(x)\\
    &\stackrel{!}{=} \Phi^{f \cdot g}_{L(\xi)}(x)\\
    &= E(f \cdot g)(\dot x \cons^* \xi),
  \end{align*}
  where all except the marked equality are easily verified.
  We will prove the marked equality by induction on the length of $\xi$.

  For the base case, note that
  \[
    \lim_{n \to \infty} \Phi^f_{L(\tuple{} \snoc (Eg \upharpoonright n))}(x) = \lim_{n \to \infty} \Phi^f_{L(\tuple{Eg \upharpoonright n})}(x) = \lim_{n \to \infty} \Phi^f_{r_{0, Eg \upharpoonright n}}(x)
  \]
  and
  \[
    \Phi^{f \cdot g}_{L(\tuple{})}(x) = \Phi^{f \cdot g}_{e_\ast}(x) = (f \cdot g)(x) = \Phi^{f \oplus g}_r(x).
  \]

  We will refer to $f$ as the left oracle and $g$ or $g \upharpoonright n$ as the right oracle,
  and refer to $\Phi^f_{r_{0, Eg \upharpoonright n}}(x)$ and $\Phi^{f \oplus g}_r(x)$ as the
  former and latter computation respectively.
  The former and latter computations will behave identically
  except if the former performs an unsuccessful oracle query
  (which is necessarily to the right oracle, since the left oracle is total) or runs out of steps,
  in which case it will terminate with result 0.
  If this does not occur and both computations terminate, they will, for some $k$,
  terminate with result $k+1$ and $k$ respectively, and hence satisfy the conditions for
  the latter to be the limit of the former, provided the computations do indeed terminate.
  This is the case, since the latter computation only performs finitely many queries to the oracle,
  and so for a sufficiently large $n$ the former computation will have no unsuccessful oracle queries,
  and the two will both run to completion.

  For the inductive step, suppose that for $\xi \in \omega$ we have already shown
  \[
    \lim_{n \to \infty} \Phi^f_{L(\xi \snoc (Eg \upharpoonright n))} = \Phi^{f \cdot g}_{L(\xi)}.
  \]

  We now wish to show that for all $\alpha \in \omega$,
  \[
    \lim_{n \to \infty} \Phi^f_{L(\alpha \cons \xi \snoc (Eg \upharpoonright n))}
      = \lim_{n \to \infty} \Phi^{\Phi^f_{L(\xi \snoc (Eg \upharpoonright n))}}_{r_{|\xi|+1, \alpha}}
      \stackrel{!}{=} \Phi^{\Phi^{f \cdot g}_{L(\xi)}}_{r_{|\xi|, \alpha}}
      = \Phi^{f \cdot g}_{L(\alpha \cons \xi)},
  \]
  where again only the innermost equality is of interest; the rest follow by expanding the definitions.

  For convenience, we will write
  \[
    h_n = \Phi^f_{L(\xi \snoc (Eg \upharpoonright n))} \text{ and } H = \Phi^{f \cdot g}_{L(\xi)}.
  \]

  Our induction hypothesis is thus $\lim_{n \to \infty} h_n = H$, and our goal, letting $m = |\xi|$, is
  to show that for all $x, \alpha \in \omega$,
  \[
    \lim_{n \to \infty} \Phi^{h_n}_{r_{m + 1, \alpha}}(x) = \Phi^H_{r_{m, \alpha}}(x).
  \]

  As in the base case, for any $x$, the two computations behave identically
  except for oracle queries and step bounds, and except that if both run to completion,
  the left will return a value one greater than the right, which
  is exactly what we want for the limit.
  We again refer to $h_n$ or $H$ as the left oracle and $\alpha$ as the right oracle,
  and to $\Phi^{h_n}_{r_{m + 1, \alpha}}(x)$ and $\Phi^H_{r_{m, \alpha}}(x)$ as the former and latter
  computations respectively.

  Let us consider the possible oracle queries:
  \begin{itemize}
    \item If the computations query the right oracle, it is either out of bounds,
      in which case both computations will terminate and will return $m+1$ and $m$ respectively,
      or it is in bounds, in which case both will see the same oracle response.
      The same occurs if the computations hit the step bounds.
    \item If the computations query the left oracle, and the former computation receives 0, then it will
      immediately terminate with result 0.
    \item If the computations query the left oracle, and the former computation receives $k + 1 < m + 1$, then the
      latter computation will receive $k < m$, and both will terminate, with result $k+1$ and $k$ respectively.
    \item If the computations query the left oracle, and the former computation receives $k + 1 \ge m + 1$, then the
      latter computation will receive $k \ge m$, and both use $k + 1 - m + 1 = k - m$ as the oracle response.
  \end{itemize}
  We thus see that unless the former computation terminates with result 0 (which does not affect the limit behavior),
  the two computations will indeed return corresponding results, with the result of the former computation one greater,
  as required by the definition of a limit.
  Again, it remains to show that the former computation will eventually give a non-zero response, and this
  is indeed the case: namely, as before, the latter computation will query the left oracle only finitely often,
  and there exists some $n \in \omega$ such that $h_n$ gives a non-zero result for those finitely many queries.

  It follows that the limit indeed holds, which concludes the inductive case.
  We thus have that $(Ef \cdot Eg)(\dot x \cons^* \xi) = E(f \cdot g)(\dot x \cons^* \xi)$, as desired,
  and hence $E$ is indeed an embedding.
\end{proof}

\begin{theorem}
  There exists an embedding $E : \PP_\bot \to \B$, which restricts to an embedding $\PP_\bot^X \to \B^X$
  for every oracle $X$.
\end{theorem}

\begin{proof}
  The structure of this proof is similar to that of Theorem~\ref{thm:p_embed_k2},
  but the details of the encoding differ considerably.
  Once again, we will first construct $E : \PP_\bot \to \B$ and then show that it is an embedding.
  Since application in $\PP_\bot$ and $\B$ is total, it suffices to show that for all $\vph, \psi \in \PP_\bot$,
  $E\vph \cdot E\psi = E(\vph \cdot \psi)$.
  
  Given $x \in \omega$, let $\dot x = \tuple{\tuple{}, 0, x}$.
  We split $\omega$ into $X$ and $Y$ as follows:
  \begin{align*}
    X &= \compr{\dot{x} \cons^* \xi}{x, \xi \in \omega}\\
    Y &= \omega - X,
  \end{align*}

  In the proof of Theorem~\ref{thm:p_embed_k2}, we used the fact that if a sequence
  of functions $f_n$ had a limit $g$, then for all $i$ and $x$, $f_i(x)$ was an
  approximation of $g(x)$.
  This is not the case here: given a family of functions $\delta_\alpha$ with
  limit $\Delta$, for every $x$ we will in general need a different $\beta$ for
  $\delta_\beta(x)$ to be an approximation of $\Delta(x)$, since $\beta$ should
  contain responses to the queries made by $\delta_{\beta'}(x)$ for $\beta' \sqsubset \beta$.
  To work around this problem, we thread the responses provided to us
  in $\xi$ through our computation, allowing subcomputations to consume some of the responses
  and then return the remaining ones.

  Given $n \in \omega$, let $p_{n, m}$ be the function that maps $\say^n(k \cons^* \xi)$ to $\say^n(k)$
  when $|\xi| = m$, and that is the identity everywhere else.
  Let $r_n$ be defined inductively as follows.
  Let $r_0 = e_\ast$, the identity of $\ast$ as defined in (\ref{eqn:ast}).
  Let $r_{n+1}$ be the code of the function that, on input $(x \cons \beta) \cons^* \xi$ where $|\xi| = n$,
  simulates $r$ on input $x$, but modifies the oracle queries as follows.
  On an oracle query of the form $2l$, run $r_n$ on input $l \cons^* \xi$.
  If the result is of the form $\say^n(k \cons^* \xi')$, update $\xi$ to be $\xi'$ and use $k$
  as the result of the query.
  If the result is of any other form $k$, terminate the computation and return $k$.
  On an oracle query of the form $2l+1$, use the first value in $\beta$, removing it from $\beta$.
  If no value is available, return $\say^n(\ask(\dot{l}))$.
  If $r$ runs to completion with result $k$, return $\say^{n+1}((k \cons \beta) \cons^* \xi)$,
  where $\beta$ and $\xi$ have possibly been updated in the process of simulating $r$.

  We define $E$ as follows:
  \begin{align*}
    (E\vph)(\dot x \cons^* \xi) &= p_{|\xi|, |\xi|}\left(\Phi^{\vph}_{r_{|\xi|}}(x \cons^* \xi)\right) & \text{for $x, \xi \in \omega$}\\ 
    (E\vph)(\tuple{y}) &= \ask(y) & \text{for $y \in Y$}\\
    (E\vph)(\tuple{y, z}) &= \say(z) & \text{for $y \in Y, z \in \omega$}\\
    (E\vph)(y) &= 0 & \text{otherwise.}
  \end{align*}

  Once again, we see that $E\vph$ is partial $\vph$-computable and that $E$ is injective, as desired.
  As before, $(E\vph)(\dot x) = \Phi^\vph_{e_\ast}(x) = \vph(x)$ whenever the right-hand side is defined.

  To show that $E$ is an embedding, it remains to verify that it preserves application.
  Let $\vph, \psi \in \PP_\bot$.

  For $y \in Y$, we see that
  \[
    (E\vph \cdot E\psi)(y) = \lim_{\alpha \to E\psi} (E\vph)(y \cons \alpha) = (E\psi)(y),
  \]
  where the first equality is by definition and the second is witnessed by the dialogue
  $\tuple{y, (E\psi)(y)}$.
  Since $y \in Y$, we have $(E\psi)(y) = E(\vph \cdot \psi)(y)$ by our definition of $E$,
  as desired.
  Note that we have no problems with non-termination here; all terms are defined.

  Let us now look at $X$.
  Our aim is to show that for all $x, \xi \in \omega$, letting $n = |\xi|$, the following sequence of congruences holds:
  \begin{align*}
    (E\vph \cdot E\psi)(\dot x \cons^* \xi)
      &\simeq \lim_{\alpha \to E\psi} p_{n+1, n+1}\left(\Phi^\vph_{r_{n+1}}((x \cons^* \xi) \cons \alpha)\right)\\
      &\simeq p_{n, n+1}\left(\lim_{\alpha \to E\psi} \Phi^\vph_{r_{n+1}}((x \cons^* \xi) \cons \alpha)\right)\\
      &\stackrel{!}{\simeq} p_{n, n}\left(\Phi^{\vph \cdot \psi}_{r_n}(x \cons^* \xi)\right)\\
      &\simeq E(\vph \cdot \psi)(\dot x \cons^* \xi).
  \end{align*}

  Only the marked equivalence is of interest, as the first and last hold by definition, and the second is easily checked.
  We show that this equivalence holds by an induction on the length of $\xi$.
  Again letting $n = |\xi|$, we show:
  \begin{itemize}
    \item If $\Phi^{\vph \cdot \psi}_{r_n}(x \cons^* \xi)$ terminates and is of the form $\say^n(k \cons^* \xi')$ then
      \[ \lim_{\alpha \to E\psi} \Phi^\vph_{r_{n+1}}((x \cons^* \xi) \cons \alpha) = \say^n((k \cons^* \xi') \cons \tuple{}), \]
      and if $\alpha$ is a supersequence of the sequence witnessing the limit, then
      \begin{equation}
        \label{eqn:ind_2}
        \Phi^\vph_{r_{n+1}}((x \cons^* \xi) \cons \alpha) = \say^{n+1}((k \cons^* \xi') \cons \alpha'),
      \end{equation}
      where $\alpha'$ is the sequence of oracle responses not consumed by the computation.
    \item If $\Phi^{\vph \cdot \psi}_{r_n}(x \cons^* \xi)$ terminates with any result $k$ not of the above form,
      then
      \[ \lim_{\alpha \to E\psi} \Phi^\vph_{r_{n+1}}((x \cons^* \xi) \cons \alpha) = k. \]
    \item If $\Phi^{\vph \cdot \psi}_{r_n}(x \cons^* \xi)$ does not terminate, then
      \[ \lim_{\alpha \to E\psi} \Phi^\vph_{r_{n+1}}((x \cons^* \xi) \cons \alpha) \]
      does not exist.
  \end{itemize}

  To see this suffices, note that if $\Phi^{\vph \cdot \psi}_{r_n}(x \cons^* \xi)$ terminates then, whatever the result, we have
  \[
      p_{n, n+1}\left(\lim_{\alpha \to E\psi} \Phi^\vph_{r_{n+1}}((x \cons^* \xi) \cons \alpha)\right)
      = p_{n, n}\left(\Phi^{\vph \cdot \psi}_{r_n}(x \cons^* \xi)\right).
  \] 

  Conversely, if $\Phi^{\vph \cdot \psi}_{r_n}(x \cons^* \xi)$ diverges then both sides diverge, as desired.

  Let us now proceed with the inductive proof.
  For the base case, note that
  \[ \Phi^{\vph \cdot \psi}_{r_0}(x) = \Phi^{\vph \cdot \psi}_{e_\ast}(x) = (\vph \cdot \psi)(x) = \Phi^{\vph \oplus \psi}_r(x) \]
  and consider the computation $\Phi^\vph_{r_1}(x \cons \alpha)$ in the context of the limit
  \begin{equation}
    \label{eq:B_base_lim} \lim_{\alpha \to E\psi} \Phi^\vph_{r_1}(x \cons \alpha).
  \end{equation}

  This computation simulates $r$, and this simulation can do the following:
  \begin{itemize}
    \item Query the oracle at $2l$: it will receive $\vph(l)$ as response, since $r_0 = e_\ast$.
      Note that the response is always of the form $k = \say^0(k \cons^* \tuple{})$.
    \item Query the oracle at $2l + 1$, with no more elements of $\alpha$ available: the computation terminates with response $\ask(\dot l)$,
      ensuring $\alpha$ will be extended with $(E\psi)(\dot l) = \psi(l)$.
    \item Query the oracle at $2l + 1$, with elements of $\alpha$ remaining: the next element of $\alpha$ is $\psi(l)$, as
      ensured by the previous point.
    \item Terminate with result $k$: the computation as a whole terminates with $\say(k \cons \alpha)$,
      where $\alpha$ may have been updated in the process of the computation.
  \end{itemize}
  We thus see that if $\alpha$ is sufficiently long and $\Phi^{\vph \oplus \psi}_r(x)$ terminates, then
  \[
    \Phi^\vph_{r_1}(x \cons \alpha) = \say\left(\Phi^{\vph \oplus \psi}_r(x) \cons \alpha'\right),
  \]
  where $\alpha'$ is the sequence of left-over elements of $\alpha$ after all oracle queries have been handled by the simulation.
  It follows that
  \[
    \lim_{\alpha \to E\psi} \Phi^\vph_{r_{n+1}}(x \cons \alpha) = \Phi^{\vph \oplus \psi}_r(x) \cons \tuple{}
  \]
  and (\ref{eqn:ind_2}) holds, as required.
  If $\Phi^{\vph \oplus \psi}_r(x)$ diverges, then the limit does not exist, satisfying the last condition of the induction.

  This concludes the base case.
  The argument for the inductive case is analogous, with the computations of interest being
  \[
    \Phi^\vph_{r_{n+2}}(((x \cons \beta) \cons^* \xi) \cons \alpha) \text{ and } \Phi^{\vph \cdot \psi}_{r_{n+1}}(((x \cons \beta) \cons^* \xi),
  \] 
  which we will refer to as the former and latter computation respectively.

  Both of these computations simulate $r$, and differ only in their behavior with respect to oracle queries,
  as well as the result upon termination.
  While there are many cases, the argument in each case is straightforward.
  The possible things the simulations can do are:
  \begin{itemize}
    \item Query the oracle at $2l$, with results of the form $\say^{n+1}((k \cons^* \xi) \cons \alpha)$
      and $\say^n(k \cons^* \xi)$ respectively: both simulations proceed with result $k$
    \item Query the oracle at $2l$, with results of the form $\say(k)$ and $k$ respectively:
      both computations terminate, with result $\say(k)$ and $k$ respectively.
    \item Query the oracle at $2l$, with a result of the form $\ask(k)$ for the former computation:
      the former computation terminates with result $\ask(k)$, ensuring $\alpha$ is extended.
    \item Query the oracle at $2l+1$, and $\beta$ has elements remaining:
      both simulations use the next element of $\beta$ as the value.
    \item Query the oracle at $2l+1$, and $\beta$ is empty:
      the computations terminate with results $\say^{n+1}(\ask(\dot{l}))$ and $\say^n(\ask(\dot{l}))$ respectively.
    \item Both simulations terminate with result $k$:
      the computations terminate with result $\say^{n+2}(((k \cons \beta) \cons^* \xi) \cons \alpha)$ and
      $\say^{n+1}((k \cons \beta) \cons^* \xi)$ respectively.
  \end{itemize}

  We see that the simulations will proceed in lockstep, and that the relation between the
  final results will satisfy the definition of a limit, as desired.
  The conditions of the induction can be checked the same way as in the base case.

  Putting it all together, we see that for all $x \in \omega$, $(E\vph \cdot E\psi)(x) = E(\vph \cdot \psi)(x)$, as desired.
  It follows that $E$ is indeed an embedding, concluding the proof.
\end{proof}

\end{document}